\documentclass[a4paper]{article}

\usepackage[utf8]{inputenc}
\usepackage{lmodern}
\usepackage[colorlinks]{hyperref}
\usepackage[english]{babel}
\usepackage[margin=22mm]{geometry}

\usepackage{tikz}
\usepackage{subcaption}

\usepackage{amssymb, amsfonts, amsmath, amsthm, mathrsfs,bbm}
\theoremstyle{plain}
\newtheorem{theorem}{Theorem}[section]
\newtheorem{corollary}[theorem]{Corollary}
\newtheorem{lemma}[theorem]{Lemma}
\newtheorem{proposition}[theorem]{Proposition}

\newtheorem{conjecture}[theorem]{Conjecture}

\theoremstyle{remark}
\newtheorem{remark}[theorem]{Remark}

\numberwithin{equation}{section}

\usepackage{graphicx}
\def\shrug{\texttt{\raisebox{0.75em}{\char`\_}\char`\\\char`\_\kern-0.5ex(\kern-0.25ex\raisebox{0.25ex}{\rotatebox{45}{\raisebox{-.75ex}"\kern-1.5ex\rotatebox{-90})}}\kern-0.5ex)\kern-0.5ex\char`\_/\raisebox{0.75em}{\char`\_}}}

\newcommand{\N}{\mathbb{N}}

\newcommand{\R}{\mathbb{R}}

\newcommand{\ind}[1]{\mathbbm{1}_{\left\{#1\right\}}}

\newcommand{\norme}[1]{{\left\Vert #1 \right\Vert}}
\newcommand{\crochet}[1]{{\left\langle #1 \right\rangle}}
\renewcommand{\bar}[1]{\overline{#1}}
\renewcommand{\tilde}[1]{\widetilde{#1}}
\renewcommand{\hat}[1]{\widehat{#1}}

\renewcommand{\phi}{\varphi}
\renewcommand{\epsilon}{\varepsilon}

\newcommand{\E}{\mathbf{E}}
\renewcommand{\P}{\mathbf{P}}

\newcommand{\dd}{\mathrm{d}}
\newcommand{\egaldistr}{{\overset{(d)}{=}}}

\renewcommand{\rho}{\varrho}

\usepackage{xcolor}

\title{Anomalous spreading in reducible multitype\\ branching Brownian motion}

\author{Mohamed Ali Belloum\thanks{\texttt{belloum@math.univ-paris13.fr}, Université Sorbonne Paris Nord, LAGA, CNRS, UMR 7539, F-93430, Villetaneuse, France.}
\and Bastien Mallein\thanks{\texttt{mallein@math.univ-paris13.fr}, Université Sorbonne Paris Nord, LAGA, UMR 7539, F-93430, Villetaneuse, France.}}

\date{\today}

\begin{document}

\maketitle

\begin{abstract}
We consider a two-type reducible branching Brownian motion, defined as a particle system on the real line in which particles of two types move according to independent Brownian motion and create offspring at constant rate. Particles of type $1$ can give birth to particles of types $1$ and $2$, but particles of type $2$ only give birth to descendants of type $2$. Under some specific conditions, Biggins \cite{Big12} shows that this process exhibit an anomalous spreading behaviour: the rightmost particle at time $t$ is much further than the expected position for the rightmost particle in a branching Brownian motion consisting only of particles of type $1$ or of type $2$. This anomalous spreading also has been investigated from a reaction-diffusion equation standpoint by Holzer \cite{Hol13,Hol16}. The aim of this article is to study the asymptotic behaviour of the position of the furthest particles in the two-type reducible branching Brownian motion, obtaining in particular tight estimates for the median of the maximal displacement.
\end{abstract}

\section{Introduction}
\label{sec:intro}

The standard \emph{branching Brownian motion} is a particle system on the real line that can be constructed as follows. It starts with a unique particle at time $0$, that moves according to a standard Brownian motion. After an exponential time of parameter $1$, this particle dies and is replaced by two children. The two daughter particles then start independent copies of the branching Brownian motion from their current position. For all $t \geq 0$, we write $\mathcal{N}_t$ the set of particles alive at time $t$, and for $u \in \mathcal{N}_t$ we denote by $X_u(t)$ the position at time $t$ of that particle.

The branching Brownian motion (or BBM) is strongly related to the \emph{F-KPP reaction-diffusion equation}, defined as
\begin{equation}
  \label{eqn:fkpp}
  \partial_t u = \frac{1}{2} \Delta u - u(1-u).
\end{equation}
More precisely, given a measurable function $f : \R \to [0,1]$, set for $x \in \R$ and $t \geq 0$
\[
  u_t(x) = \E \left( \prod_{u \in \mathcal{N}_t} f(X_u(t) + x) \right),
\]
then $u$ is the solution of \eqref{eqn:fkpp} with $u_0(x) = f(x)$. In particular, setting $M_t = \max_{u \in \mathcal{N}_t} X_u(t)$, we note that the tail distribution of $-M_t$ is the solution at time $t$ of \eqref{eqn:fkpp} with $u_0(z) = \ind{z < 0}$.

Thanks to this observation, Bramson \cite{Bra78} obtained an explicit formula for the asymptotic behaviour of the median of $M_t$. Precisely, he observed that setting
\begin{equation}
  \label{eqn:defMedia}
 m_t:= \sqrt{2} t - \frac{3}{2\sqrt{2}} \log t
\end{equation}
the process $(M_t- m_t, t \geq 0)$ is tight. Lalley and Sellke \cite{LaS87} refined this result and prove that $M_t - m_t$ converges in law toward a Gumbel random variable shifted by an independent copy of $\frac{1}{\sqrt{2}}\log Z_\infty$, where $Z_\infty$ is the a.s. limit as $t \to \infty$ of the derivative martingale, defined by
\[
  Z_t  := \sum_{u \in \mathcal{N}_t} (\sqrt{2} t - X_u(t)) e^{\sqrt{2} X_u(t) - 2t}.
\]

The derivative martingale is called that way due to its relationship to the derivative at its critical point of the additive martingale introduced by McKean in \cite{McK}, defined as
\[
  W_t(\theta) := \sum_{u \in \mathcal{N}_t} e^{\theta X_u(t) - t(1 + \frac{\theta^2}{2})}.
\]
It was shown in \cite{Nev} that $(W_t(\theta), t \geq 0)$ is uniformly integrable if and only if $|\theta| < \sqrt{2}$ and converges to an a.s. positive limit $W_\infty(\theta)$ in that case. Otherwise, it converges to $0$ a.s. This result has later been extended by Biggins \cite{Big} and Lyons \cite{Lyo97} to the branching random walk, which is a discrete-time analogous to the BBM.

The behaviour of the particles at the tip of branching Brownian motions was later investigated by Aidékon, Berestycki, Brunet and Shi \cite{ABBS} as well as Arguin, Bovier and Kistler \cite{ABK1,ABK2,ABK3}. They proved that the centred \emph{extremal process} of the standard BBM, defined by
\[
  \hat{\mathcal{E}}^R_t = \sum_{u \in \mathcal{N}_t} \delta_{X_u(t)-m_t}
\]
converges in distribution to a decorated Poisson point process with (random) intensity $\sqrt{2}c_\star Z_\infty e^{-\sqrt{2} x} \dd x$. More precisely, there exists a law $\mathfrak{D}$ on point measures such that writing $(D_j, j \geq 1)$ i.i.d. point measures with law $\mathfrak{D}$ and $(\xi_j, j \geq 0)$ the atoms of an independent Poisson point process with intensity $\sqrt{2}c_\star e^{-\sqrt{2}x}\dd x$, which are further independent of $Z_\infty$, and defining
\[
  \mathcal{E}_\infty = \sum_{j \geq 1} \sum_{d \in D_j} \delta_{\xi_j + d + \frac{1}{\sqrt{2}} \log Z_\infty},
\]
we have $\lim_{t \to \infty} \mathcal{E}_t = \mathcal{E}_\infty$ in law, for the topology of the vague convergence. We give more details on these results in Section~\ref{sec:facts}.

We refer to the above limit as a decorated Poisson point process, or $\text{DPPP}(\sqrt{2}c_\star Z_\infty e^{-\sqrt{2} x} \dd x,\mathfrak{D})$. Maillard \cite{Mai13} obtained a characterization of this type of point processes as satisfying a stability by superposition property. This characterization was used in \cite{Mad17} to prove a similar convergence in distribution to a DPPP for the shifted extremal process of the branching random walk. Subag and Zeitouni \cite{SubZei} studied in more details the family of shifted randomly decorated Poisson random measures with exponential intensity.

In this article, we take interest in the \emph{two-type reducible branching Brownian motion}. This is a particle system on the real line in which particles possess a type in addition with their position. Particles of type~$1$ move according to Brownian motions with diffusion coefficient $\sigma^2_1$ and branch at rate $\beta_1$ into two children of type $1$. Additionally, they give birth to particles of type $2$ at rate $\alpha$. Particles of type $2$ move according to Brownian motions with diffusion coefficient $\sigma^2_2$ and branch at rate $\beta_2$, but cannot give birth to descendants of type $1$.

In \cite{Big12}, Biggins observe that in some cases multitype reducible branching random walks exhibit an \emph{anomalous spreading} property. Precisely, the rightmost particle at time $t$ is shown to be around position $vt$, with the speed $v$ of the two-type process being larger than the speed of a branching random walk consisting only of particles of type $1$ or uniquely of particles of type $2$. Therefore, the multitype system can invade its environment at a higher speed than the one that either particles of type $1$ or particles of type $2$ would be able to sustain on their own.

Holzer \cite{Hol13,Hol16} extended the results of Biggins to this setting, by considering the associated system of F-KPP equations, describing the speed of the rightmost particle in the system in terms of $\sigma_1,\beta_1, \sigma_2$ and $\beta_2$ (the parameter $\alpha$ does not modify the speed of the two-type particle system). Our aim is to study in more details the position of the maximal displacement, in particular in the case when anomalous spreading occurs, for this two type BBM. We also take interest in the extremal process formed by the particles of type~$2$ at time $t$, and show it to converge towards a DPPP.

Recall that the reducible two-type BBM is defined by five parameters, the diffusion coefficient $\sigma_1^2$, $\sigma_2^2$ of particles of type $1$ and $2$, their branching rate $\beta_1$, $\beta_2$, and the rate $\alpha$ at which particles of type $1$ create particles of type $2$. However, up to a dilation of time and space, it is possible to modify these parameters in such a way that $\sigma_2^2 = \beta_2 = 1$. Additionally, the parameter $\alpha$ plays no role in the value of the speed of the multitype process. We can therefore describe the phase space of this process in terms of the two parameters $\sigma^2 := \sigma_1^2$ and $\beta = \beta_1$, and identify for which parameters does anomalous spreading occurs. This is done in Figure~\ref{figure}.

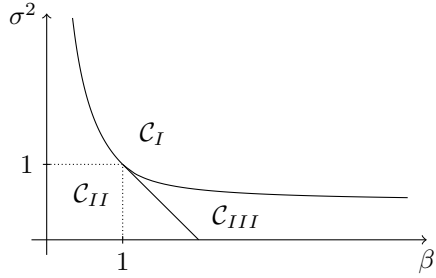
\begin{figure}[ht]
\centering
\begin{tikzpicture}[scale=1]
  \draw[->] (-0.2,0) -- (5,0) node[below] {$\beta$};
  \draw[->] (0,-0.2) -- (0,3) node[left] {$\sigma^2$};
  \draw (-0.05,1) node[left] {1} -- (0.05,1);
  \draw (1,-0.05) node[below] {1} -- (1,0.05);
  \draw [densely dotted] (0,1) -- (1,1) -- (1,0);
  \draw [domain=1:4.75,smooth,variable=\x] plot ({\x},{\x/(2*\x-1)});
  \draw [domain=1:2,smooth,variable=\x] plot ({\x},{2-\x});
  \draw [domain=0.34:1,smooth,variable=\x] plot ({\x},{1/\x});
  \draw (0.6,0.6) node{$\mathcal{C}_{II}$};
  \draw (1.4,1.4) node{$\mathcal{C}_{I}$};
  \draw (2.5,0.3) node{$\mathcal{C}_{III}$};
\end{tikzpicture}
\caption{Phase diagram of the two-type reducible BBM.}
\label{figure}
\end{figure}

We decompose the state space $(\beta,\sigma^2) \in \R_+^2$ into three regions:
\begin{align*}
  \mathcal{C}_{I} &= \left\{ (\beta,\sigma^2) : \sigma^2 > \frac{\ind{\beta \leq 1}}{\beta} + \ind{\beta > 1} \frac{\beta}{2\beta-1}\right\} \\
  \mathcal{C}_{II} &= \left\{ (\beta,\sigma^2) : \sigma^2 < \frac{\ind{\beta \leq 1}}{\beta} + \ind{\beta > 1}(2-\beta)\right\} \\
  \mathcal{C}_{III} &= \left\{ (\beta,\sigma^2) : \sigma^2 + \beta > 2 \text{ and } \sigma^2 < \frac{\beta}{2\beta -1} \right\}.
\end{align*}
If $(\beta,\sigma^2) \in \mathcal{C}_I$, the speed of the two-type reducible BBM is $\sqrt{2\beta \sigma^2}$, which is the same as particles of type $1$ alone, ignoring births of particles of type $2$. Thus in this situation, the asymptotic behaviour of the extremal process is dominated by the long-time behaviour of particles of type $1$. Conversely, if $(\beta,\sigma^2) \in \mathcal{C}_{II}$, then the speed of the process is $\sqrt{2}$, equal to the speed of a single BBM of particles of type $2$. In that situation, the asymptotic behaviour of particles of type $2$ dominates the extremal process. Finally, if $(\beta,\sigma^2) \in \mathcal{C}_{III}$, the speed of the process is larger than $\max(\sqrt{2},\sqrt{2\beta\sigma^2})$, and we will show that in this case the extremal process will be given by a mixture of the long-time asymptotic of the processes of particles of type $1$ and $2$. 

For all $t \geq 0$, we write $\mathcal{N}_t$ the set of all particles alive at time $t$, as well as $\mathcal{N}^1_t$ and $\mathcal{N}^2_t$ the set of particles of type $1$ and type $2$ respectively. We also write $X_u(t)$ for the position at time $t$ of $u \in \mathcal{N}_t$, and for all $s \leq t$, $X_u(s)$ for the position of the ancestor at time $s$ of particle $u$. If $u \in \mathcal{N}^2_t$, we denote by $T(u)$ the time at which the oldest ancestor of type $2$ of $u$ was born. In this article, we study the asymptotic behaviour of the extremal process of particles of type $2$ in this $2$-type BBM for Lebesgue-almost every values of $\sigma^2$ and $\beta$ (and for $\alpha \in (0,\infty)$).

We divide the main result of our article into three theorems, one for each area the pair $(\beta,\sigma^2)$ belongs to. We begin with the asymptotic behaviour of extremal particles when $(\beta,\sigma^2) \in \mathcal{C}_I$, in which case the extremal point measure is similar to the one observed in a branching Brownian motion of particles of type $1$.
\begin{theorem}[Domination of particles of type $1$]
\label{thm:mainI}
If $(\beta,\sigma^2) \in \mathcal{C}_I$, then there exist a constant $c_{(I)}>0$ and a point measure distribution $\mathfrak{D}^{(I)}$ such that setting  $m^{(I)}_t:= \sqrt{2\sigma^2\beta} t - \frac{3}{2\sqrt{2\beta/\sigma^2}} \log t$ we have
\[
  \lim_{t \to \infty} \sum_{u \in \mathcal{N}_t^2} \delta_{X_u(t) - m^{(I)}_t} = \mathcal{E}^{(I)}_\infty \quad \text{ in law},
\]
for the topology of the vague convergence, where $\mathcal{E}^{(I)}_\infty$ is a DPPP($\sqrt{2 \beta /\sigma^2}c_{(I)} Z^{(1)}_\infty e^{-\sqrt{2\beta/\sigma^2} x}\dd x, \mathfrak{D}^{(I)}$), where
\[
  Z^{(1)}_\infty := \lim_{t \to \infty} \sum_{u \in \mathcal{N}^1_t} (\sqrt{2 \sigma^2 \beta} t - X_u(t)) e^{\sqrt{2\beta/\sigma^2} X_u(t) - 2\beta t} \quad \text{a.s.}
\]
Additionally, we have $\displaystyle \lim_{t \to \infty} \P(M_t \leq m^{(I)}_t + x) = \E\left( e^{-c_{(I)} Z^{(1)}_\infty e^{-\sqrt{2 \beta /\sigma^2} x}} \right)$ for all $x \in \R$.
\end{theorem}

We underscore that in this theorem, the values of $c_{(I)}$ and $\mathfrak{D}^{(I)}$ are obtained implicitly, and depend on the parameters $\alpha$, $\beta$ and $\sigma^2$ of the multitype branching Brownian motion. The identification of the law of the extremal point measure $\mathcal{E}^{(I)}$ is based on the fact that it satisfies a stability by superposition property, and using Maillard's characterization \cite{Mai13}.

If $(\beta,\sigma^2) \in \mathcal{C}_{II}$, we show that the extremal process of the two-type BBM is similar to the extremal process of a single BBM of particles of type $2$, up to a random shift whose law depend on the behaviour of particles of type $1$.
\begin{theorem}[Domination of particles of type $2$]
\label{thm:mainII}
If $(\beta,\sigma^2) \in \mathcal{C}_{II}$, then writing $c_\star > 0$ the prefactor of the intensity measure in the extremal process of the standard BBM and $\mathfrak{D}$ the law of its decoration, setting $m^{(II)}_t:= m_t = \sqrt{2} t - \frac{3}{2\sqrt{2}} \log t$, we have
\[
  \lim_{t \to \infty} \sum_{u \in \mathcal{N}_t^2} \delta_{X_u(t) - m^{(II)}_t} = \mathcal{E}^{(II)}_\infty \quad \text{ in law},
\]
for the topology of the vague convergence, where $\mathcal{E}^{(II)}_\infty$ is a DPPP($\sqrt{2} c_\star \bar{Z}_\infty e^{-\sqrt{2} x}\dd x, \mathfrak{D}$) and $\bar{Z}_\infty$ is defined in Lemma~\ref{lem:limitII}.  Additionally, for all $x \in \R$ we have $\displaystyle \lim_{t \to \infty} \P(M_t \leq m^{(II)}_t + x) = \E\left( e^{-c_\star \bar{Z}_\infty e^{-\sqrt{2} x}} \right).$
\end{theorem}

We finally take interest in the case $(\beta,\sigma^2) \in \mathcal{C}_{III}$, in this situation the BBM exhibits an anomalous spreading behaviour. The extremal process contains only particles of type $2$, but particles travel at greater speed that would have been observed in a BBM of particles of type $1$ or of type $2$.
\begin{theorem}[Anomalous spreading]
\label{thm:mainIII}
If $(\beta,\sigma^2) \in \mathcal{C}_{III}$, then setting
\[
  m^{(III)}_t= \frac{\sigma^2 - \beta}{\sqrt{2(1-\sigma^2)(\beta- 1)}} t \quad \text{and} \quad \theta = \sqrt{2 \frac{\beta-1}{1-\sigma^2}},
\]
we have
\[ 
  \lim_{t \to \infty} \sum_{u \in \mathcal{N}_t^2} \delta_{X_u(t) - m^{(III)}_t} = \mathcal{E}^{(III)}_\infty \quad \text{ in law},
\]
for the topology of the vague convergence, where $\mathcal{E}_\infty$ is a DPPP($\theta c_{(III)} W_\infty(\theta) e^{-\theta x} \dd x, \mathfrak{D}^{(III)}$) and
\begin{itemize}
  \item $W_\infty(\theta) = \lim_{t \to \infty} \sum_{u \in \mathcal{N}^1_t} e^{\theta X_u(t) - t (\beta + \theta^2\sigma^2/2)}$ is the a.s. limit of an additive martingale of the BBM of particles of type $1$ with parameter~$\theta$,
  \item $c_{(III)} = \frac{\alpha C(\theta)}{2(\beta - 1)}$ with the function $C$ being defined in \eqref{eqn:largeDevMax},
  \item $\mathfrak{D}^{(III)}$ is the law of the point measure $\mathcal{D}^\theta$ defined in \eqref{eqn:defineSupercriticalDecoration}.
\end{itemize}
Additionally, for all $x \in \R$ we have $\displaystyle \lim_{t \to \infty} \P(M_t \leq m^{(III)}_t + x) = \E\left( e^{-c_{(III)} W_\infty(\theta) e^{-\theta x}} \right).$
\end{theorem}

\begin{remark}
Contrarily to what happens in Theorems~\ref{thm:mainI} and~\ref{thm:mainII}, the extremal process obtained in Theorem~\ref{thm:mainIII} is not shifted by a random variable associated to a derivative martingale, but by an additive martingale of the BBM. Additionally, it is worth noting that contrarily to the median of the maximal displacements in domains $\mathcal{C}_I$ and $\mathcal{C}_{II}$, when anomalous spreading occurs there is no logarithmic correction in the median of the maximal displacement.
\end{remark}

\begin{remark}
\label{rem:15}
Observe that in Theorems~\ref{thm:mainI}--\ref{thm:mainIII}, we obtain the convergence in law for the topology of the vague convergence of extremal processes to DPPPs as well as the convergence in law of their maximum to the maximum of this DPPP. These two convergences can be synthesized into the joint convergence of the extremal process together with its maximum, which is equivalent to the convergence of $\crochet{\mathcal{E}_t,\phi}$ for all continuous bounded functions $\phi$ with bounded support on the left (see e.g. \cite[Lemma~4.4]{BBCM19}).
\end{remark}

The rest of the article is organized as follows. We discuss our results in the next section, by putting them in the context of the state of the art for single type and multitype branching processes, and for coupled reaction-diffusion equations. In Section~\ref{sec:facts}, we introduce part of the notation and results on branching Brownian motions that will be needed in our proofs, in particular for the definition of the decorations laws of the extremal process. We introduce in Section~\ref{sec:mto} a multitype version of the celebrated many-to-one lemma. Finally, we prove Theorem~\ref{thm:mainII} in Section~\ref{sec:proofII}, Theorem~\ref{thm:mainI} in Section~\ref{sec:proofI} and Theorem~\ref{thm:mainIII} in Section~\ref{sec:anomalous}.

\section{Discussion of our main result}
\label{sec:discuss}

We compare our main results for the asymptotic behaviour of the two-type reducible branching Brownian motion to the pre-existing literature. We begin by introducing the optimization problem associated to the computation of the speed of the rightmost particle in this process. Loosely speaking, this optimization problem is related to the ``choice'' of the time between $0$ and $t$ and position at which the ancestral lineage leading to one of the rightmost positions switches from type~$1$ to type~$2$. The optimization problem was introduced by Biggins \cite{Big12} for the computation of the speed of multitype reducible branching random walks. It allows us to describe the heuristics behind the main theorems.

We then compare Theorem~\ref{thm:mainIII} to the results obtained on the extremal process of time-inhomogeneous branching Brownian motions, and in particular with the results of Bovier and Hartung \cite{BoH14}. In Section~\ref{subsec:pde}, we apply our results to the work of Holzer \cite{Hol13,Hol16} on coupled F-KPP equations. We end this section with the discussion of further questions of interest for multitype reducible BBMs and some conjectures and open questions.

\subsection{Associated optimization problem and heuristic}
\label{subsec:stateoftheart}

Despite the fact that spatial multitype branching processes have a long history, the study of the asymptotic behaviour of their largest displacement has not been considered until recently. As previously mentioned, Biggins \cite{Big12} computed the speed of multitype reducible branching random walks. This process is a discrete-time particles system in which each particle gives birth to offspring in an independent fashion around its position, with a reproduction law that depends on its type, under the assumption that the Markov chain associated to the type of a typical individual in the process is reducible. Ren and Yang \cite{ReY14} then considered the asymptotic behaviour of the maximal displacement in an \emph{irreducible} multitype BBM.

In \cite{BigBO}, Biggins gives an explicit description of the speed of a reducible two-type branching random walk as the solution of an optimization problem. In the context of the two-type BBM we consider, the optimization problem can be expressed as such:
\begin{equation}
  \label{eqn:optimizationProblem}
  v = \max\left\{ p a + (1-p) b : p \in [0,1],\  p \left(\frac{a^2}{2\sigma^2}-\beta\right) \leq 0,\  p \left(\frac{a^2}{2\sigma^2}-\beta\right) + (1-p) \left(\frac{b^2}{2} - 1\right) \leq 0 \right\}.
\end{equation}
This optimization problem can be understood as follows. It is well-known that if $a < \sqrt{2\sigma^2 \beta}$ and $b \geq \sqrt{2}$, there are with high probability around $e^{pt (\beta - a^2/2\sigma^2) +o(t)}$ particles of type~$1$ at time $pt$ to the right of position $pta$, and a typical particle of type~$2$ has probability $e^{(1-p)t(1 - b^2/2) +o(t)}$ of having a descendant to the right of position $(1-p)bt$ at time $(1-p)t$. Therefore, for all $(p,a,b)$ such that
\[
   p \in [0,1],\  p \left(\frac{a^2}{2\sigma^2}-\beta\right) \leq 0,\  p \left(\frac{a^2}{2\sigma^2}-\beta\right) + (1-p) \left(\frac{b^2}{2} - 1\right) \leq 0 ,
\]
by law of large numbers there should be with high probability particles of type~$2$ to the right of the position $t(pa + (1-p)b)$ at time $t$.

If we write $(p^*,a^*,b^*)$ the triplet optimizing the problem \eqref{eqn:optimizationProblem}, it follows from classical optimization under constraints computations that:
\begin{enumerate}
  \item If $(\beta,\sigma^2) \in \mathcal{C}_I$, then $p^*=1$ and $a^* = \sqrt{2\beta \sigma^2}$, which is in accordance with Theorem~\ref{thm:mainI}, as the extremal particle system is dominated by the behaviour of particles of type~$1$, and particles of type~$2$ contributing to the extremal process are close relatives descendants of a parent of type~$1$;
  \item If $(\beta,\sigma^2) \in \mathcal{C}_{II}$, then $p^*=0$ and $b^* = \sqrt{2}$, which is in accordance with Theorem~\ref{thm:mainII}, as the extremal particle system is dominated by the behaviour of particles of type~$2$, that are born at time $o(t)$ from particles of type~$1$;
  \item If $(\beta,\sigma^2) \in \mathcal{C}_{III}$, then
  \[
    p^* = \frac{\sigma^2 + \beta - 2}{2(1-\sigma^2)(\beta-1)} , \quad a^* = \sigma^2 \sqrt{2 \frac{\beta-1}{1-\sigma^2}} \ \text{  and  } \ b^* = \sqrt{2 \frac{\beta-1}{1-\sigma^2}}.
  \]
  We then have $v = \frac{\beta - \sigma^2}{\sqrt{2(1-\sigma^2)(\beta- 1)}}$, which corresponds to the main result of Theorem~\ref{thm:mainIII}. Additionally, the Lagrange multiplier associated to this optimization problem is $\theta = \sqrt{2 \frac{\beta-1}{1-\sigma^2}} = b = \frac{a}{\sigma^2}$.
\end{enumerate}

In particular, the optimization problem associated to the case $(\beta,\sigma^2) \in \mathcal{C}_{III}$ can be related to the following interpretation of Theorem~\ref{thm:mainIII}. The extremal process at time $t$ is obtained as the superposition of the extremal processes of an exponentially large number of BBMs of type~$2$, starting around time $tp^*$ and position $t p^* a^*$. The number of these BBMs is directly related to the number of particles of type~$1$ that displace at speed $a^*$, which is known to be proportional to $W_\infty(\theta) e^{t(1 - (a^*)^2/2 \sigma^2)}$. It explains the apparition of this martingale in Theorem~\ref{thm:mainIII}, whereas the decoration distribution $\mathfrak{D}^{(III)}$ is the extremal process of a BBM of type~$2$ conditionally on moving at the speed $b^* > v$.

For Theorem~\ref{thm:mainI}, a similar description can be made. We expect the asymptotic behaviour to be driven by the behaviour of particles of type~$1$, therefore the extremal process of particles of type~$2$ should be obtained as a decoration of the extremal process of particles of type~$1$. However, as we were not able to use result of convergence of extremal processes together with a description of the behaviour of particles at times $t - O(1)$, we do not obtain an explicit value for $c_{(I)}$ and an explicit description of the law $\mathfrak{D}^{(I)}$. However, with similar techniques as the ones used in \cite{ABBS} or \cite{ABK3}, such explicit constructions should be available.

Finally, in the case covered by Theorem~\ref{thm:mainII}, the above optimization problem indicates that the extremal process of the multitype reducible BBM should be obtained as the superposition of a finite number of BBMs of particles of type~$2$, descending from the first few particles of type~$2$ to be born. The random variable $\bar{Z}$ is then constructed as the weighted sum of i.i.d. copies of the derivative martingale of a standard BBM and the decoration is the same as the decoration of the original BBM.

To prove Theorems~\ref{thm:mainI}--\ref{thm:mainIII}, we show that the above heuristic holds, i.e. that with high probability the set of particles contributing to the extremal processes are the one we identified in each case. We then use previously known results of branching Brownian motions to compute the Laplace transforms of the extremal point measures we are interested in.

The solution of the optimization problem \eqref{eqn:optimizationProblem} is also solution of  $v = \sup\{a \in \R : g(a) \leq 0\}$, where $g$ is the largest convex function such that
\[
 \forall |x| \leq \sqrt{2\beta \sigma^2},\  g(x) \leq \left(\frac{x^2}{2\sigma^2} - \beta\right) \quad \text{ and } \quad \forall y \in \R, \ g(y) \leq \frac{y^2}{2} - 1,
\]
see \cite{BigBO} for precisions. The function $x \mapsto \frac{x^2}{2\sigma^2} - \beta$ is known as the rate function for particles of type~$1$, and $y \mapsto \frac{y^2}{2} - 1$ is the rate function for particles of type~$2$.

We then observe that the three cases described above are the following:
\begin{enumerate}
  \item If $(\beta,\sigma^2) \in \mathcal{C}_I$, then $v = \sqrt{2\beta \sigma^2} = \sup\{x \in \R :x^2/2\sigma^2 - \beta \leq 0\}$.
  \item If $(\beta,\sigma^2) \in \mathcal{C}_{II}$, then $v = \sqrt{2 } = \sup\{y \in \R : \frac{y^2}{2} - 1 \leq 0\}$.
  \item If $(\beta,\sigma^2) \in \mathcal{C}_{III}$, then $v > \max (\sqrt{2\beta \sigma^2}, \sqrt{2})$.
\end{enumerate}
In other words, the anomalous spreading corresponds to the case when the convex envelope $g$ crosses the $x$-axis to the right of the rate functions of particles of type~$1$ and~$2$.

\begin{figure}[ht]
\centering
\begin{subfigure}[h]{0.32\textwidth}
\centering
\begin{tikzpicture}[xscale=.8,yscale=.7]
  \draw [->] (-2.5,0) -- (2.5,0);
  \draw [->] (0,-3) -- (0,2.5);
  \draw [yellow!90!black, very thick, domain = -2.5:2.5] plot ({\x}, {\x*\x/2-1}) ;
  \draw [blue!60, very thick, domain = -2:2] plot ({\x}, {(\x*\x/(2*5/2)-.8)}) ;
  
  \draw [green!70!black, thick, domain =-.587:.587] plot ({\x}, {\x*\x/2-1});
  \draw [green!70!black, thick, domain =-2:-1.468] plot ({\x}, {(\x*\x/(2*5/2)-.8)});
  \draw [green!70!black, thick, domain =1.468:2] plot ({\x}, {(\x*\x/(2*5/2)-.8)});
  \draw [green!70!black, thick] (.587,-.828) -- (1.468,-.369) ;
  \draw [green!70!black, thick] (2,0) -- (2.5,1.707);
  \draw [green!70!black, thick] (-.587,-.828) -- (-1.468,-.369) ;
  \draw [green!70!black, thick] (-2,0) -- (-2.5,1.707);
  
  \draw (2,-.1) node[below right] {$v$} -- (2,.1);
\end{tikzpicture}
\caption{Case I: The speed of the multitype process is the same as the speed of the process of type 1 particles.}
\end{subfigure}
\hspace{\stretch{1}}
\begin{subfigure}[h]{0.32\textwidth}
\centering
\begin{tikzpicture}[xscale=.8,yscale=.7]
  \draw [->] (-2.5,0) -- (2.5,0);
  \draw [->] (0,-3) -- (0,2.5);
  \draw [yellow!90!black, very thick, domain = -2.5:2.5] plot ({\x}, {\x*\x/2-1}) ;
  \draw [blue!60, thick, domain = -.866:.866] plot ({\x}, {(\x*\x/(2*0.25)-1.5)}) ;
  
  \draw [green!70!black, domain =-.25:.25] plot ({\x}, {(\x*\x/(2*0.25)-1.5)}) ;
  \draw [green!70!black, thick, domain =-2.5:-1] plot ({\x}, {\x*\x/2-1});
  \draw [green!70!black, thick, domain =1:2.5] plot ({\x}, {\x*\x/2-1});
  \draw [green!70!black, thick] (.25,-1.375) -- (1,-.5) ;
  \draw [green!70!black, thick] (-.25,-1.375) -- (-1,-.5) ;

  \draw (1.414,-.1) node[below right] {$v$} -- (1.414,.1);
\end{tikzpicture}
\caption{Case II: The speed of the multitype process is the same as the speed of the process of type 2 particles.}
\end{subfigure}
\hspace{\stretch{1}}
\begin{subfigure}[h]{0.32\textwidth}
\centering
\begin{tikzpicture}[xscale=.8,yscale=.7]
  \draw [->] (-2.5,0) -- (2.5,0);
  \draw [->] (0,-3) -- (0,2.5);
  \draw [yellow!90!black, very thick, domain = -2.5:2.5] plot ({\x}, {\x*\x/2-1}) ;
  \draw [blue!60, thick, domain = -1.118:1.118] plot ({\x}, {(\x*\x/.5-2.5)}) ;
  \draw [green!70!black, thick] (.5,-2) -- (2,1) ;
  \draw [green!70!black, thick] (-.5,-2) -- (-2,1) ;
  
  \draw [green!70!black, thick, domain =-.5:.5] plot ({\x}, {(\x*\x/.5-2.5)});
  \draw [green!70!black, thick, domain =-2.5:-2] plot ({\x}, {(\x*\x/2-1)});
  \draw [green!70!black, thick, domain =2:2.5] plot ({\x}, {(\x*\x/2-1)});
  \draw (1.5,-.1) node[below right] {$v$} -- (1.5,.1);
\end{tikzpicture}
\caption{Case III: The anomalous spreading of the multitype process which is faster than the process consisting only of particles of type 1 or 2.}
\end{subfigure}
\caption{Convex envelope $g$ for $(\beta,\sigma^2)$ in any of the three domains of interest. The rate function of particles of type 1 and 2 are drawn in blue and yellow respectively, the function $g$ of the multitype branching process is drawn in green.}
\end{figure}
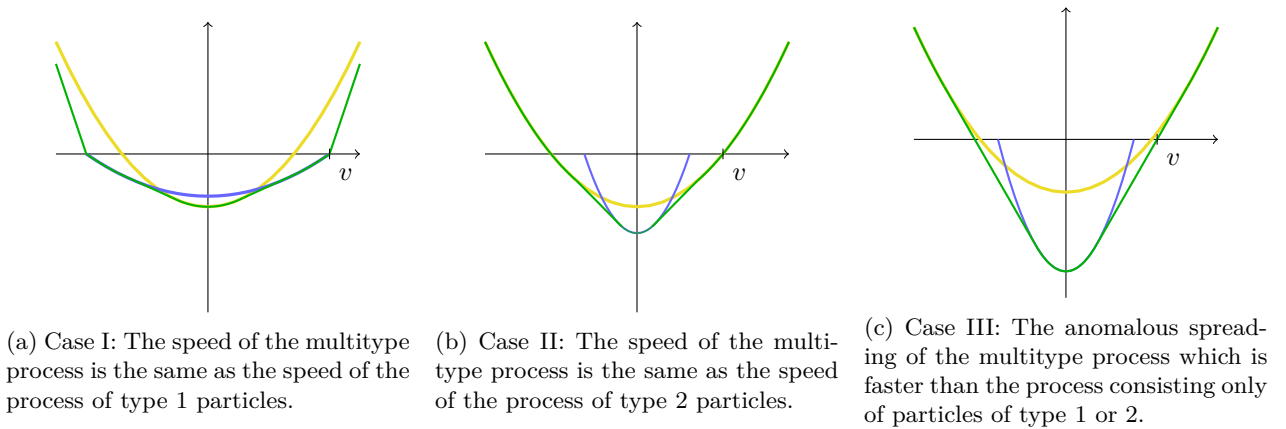

As mentioned above, Ren and Yang \cite{ReY14} studied the asymptotic behaviour of irreducible multitype BBM, and computed the speed at which that process invades its environment. In that case (i.e. when for all pair of types $i$ and $j$, individuals of type $i$ have positive probability of having at least one descendant of type $j$ after some time), this asymptotic behaviour is similar to the one obtained for a single-type BBM, with branching rate and variance obtained by considering the invariant measure of the Markov process describing the type of a typical individual. The notion of anomalous spreading in this case is thus very different, and the ancestral lineage of typical particles in the extremal process will present regular changes of type. As a result, we do not expect an asymptotic behaviour similar to the one observed in Theorem~\ref{thm:mainIII} to occur in irreducible multitype BBM.

In a different direction, Smadi and Vatutin \cite{SmV16} studied the limit in distribution of a critical reducible Galton-Watson process. It is worth noting that similarly to our results, they obtained three different behaviours for the system, with either the domination of particles of the first type, of the second type, or an interplay between the two.

\subsection{Relation to time-inhomogeneous branching processes}
\label{subsec:tiebrw}

The results presented here, in particular in the anomalous spreading case, are reminiscent of the known asymptotic for the extremal process of time-inhomogeneous branching Brownian motions. This model was introduced by Fang and Zeitouni \cite{FaZ}, and is defined as follows. Given $t \geq 0$, the process is a BBM consisting only of particles of type~$1$ until time $t/2$, at which time they all become simultaneously particles of type~$2$. It has been showed \cite{FaZ2,Mal15a} that depending on the value of $(\beta,\sigma^2)$, the position of the maximal displacement at time $t$ can exhibit different types of asymptotic behaviours. In particular, the logarithmic correction exhibit a strong phase transition in the phase space of $(\beta,\sigma^2)$.

Looking more closely at the convergence of the extremes, Bovier and Hartung \cite{BoH14,BoH15} obtained the convergence in distribution of the extremal process of the time-inhomogeneous BBM. In particular, for a multitype BBM with parameters $(\beta,\sigma^2) \in \mathcal{C}_{III}$ such that particles change from type~$1$ to type~$2$ at time $p^*t$, they showed that the extremal process converges towards $\mathcal{E}^{(III)}_\infty$, with an extra $\frac{1}{2\theta} \log t$ logarithmic correction for the centring. This is in accordance with our heuristic as we expect that the particles contributing to the extremal process at time $t$ to have been born from particles of type~$1$ around time $p^*t$.

Generalized versions of time-inhomogeneous BBM have been studied, in which the variance of particles evolves continuously over time \cite{MaZ,Mal15b}. In that case, the maximal displacement grows at constant speed with a negative correction of order $t^{1/3}$. It would be interesting to construct a multitype BBM, possibly with an infinite number of types, that would exhibit a similar phenomenon.

\subsection{F-KPP type equation associated to the multitype branching Brownian motion}
\label{subsec:pde}

Observe that similarly to the standard BBM, the multitype BBM can be associated to a reaction diffusion equation in the following way. Let $f,g : \R \to [0,1]$ be measurable functions, we define for all $x \in \R$:
\begin{align*}
  u(t,x) &= \E^{(1)}\left( \prod_{u \in \mathcal{N}^1_t} f( X_u(t)+x) \prod_{u \in \mathcal{N}^2_t} g( X_u(t)+x) \right)\\
  v(t,x) &= \E^{(2)}\left( \prod_{u \in \mathcal{N}^1_t} f( X_u(t)+x) \prod_{u \in \mathcal{N}^2_t} g( X_u(t)+x) \right)
   =  \E^{(2)}\left(\prod_{u \in \mathcal{N}^2_t} g( X_u(t)+x) \right)
\end{align*}
where $\P^{(1)}$ (respectively $\P^{(2)}$) is the law of the multitype BBM starting from one particle of type~$1$ (resp.~$2$), and we use the fact that particles of type~$2$ only produce offspring of type~$2$, with the usual convention $\prod_{u \in \emptyset} f( X_u(t)+x)  = 1$.

As under $\P^{(2)}$, the process behaves as a standard BBM, the function $v$ is a solution of the classical F-KPP reaction-diffusion equation
\begin{equation}
  \label{eqn:fkppdiff}
  \partial_t v = \frac{1}{2} \Delta v - v (1-v) \quad \text{with } v(0,x) = g(x).
\end{equation}
To obtain the partial differential equation satisfied by $u$, we observe that under law $\P^{(1)}$ one of the three following events might happen during the first $\dd t$ units of time:
\begin{itemize}
  \item with probability $\beta \dd t +o(\dd t)$, the original particle of type~$1$ branches into two offspring of type~$1$ that start i.i.d. processes with law $\P^{(1)}$;
  \item with probability $\alpha \dd t + o(\dd t)$ the particle of type~$1$ branches into one offspring of type~$1$ and one of type~$2$, that start independent processes with law $\P^{(1)}$ and $\P^{(2)}$ respectively;
  \item with probability $1 - (\beta+\alpha) \dd t$, the particle of type $2$ diffuses as the Brownian motion $\sigma B$ with diffusion constant $\sigma^2$
\end{itemize}
As a result, we have
\begin{align*}
  u(t+\dd t, x)
  &= \beta \dd t u(t,x)^2 + \alpha \dd t u(t,x)v(t,x) + (1 - (\beta + \alpha) \dd t) \E\left( u(t,x-\sigma B_{\dd t}) \right) + o(\dd t)\\
  &= u(t,x) + \dd t \left( \frac{\sigma^2}{2} \Delta u(t,x) - \beta u(1-u) - \alpha u(1-v) \right).
\end{align*}
This, together with \eqref{eqn:fkppdiff} show that $(u,v)$ is a solution of the following coupled F-KPP equation
\begin{equation}
  \label{eqn:fkppmulti}
  \begin{cases}
    & \partial_t u = \frac{\sigma^2}{2} \Delta  - \beta u(1-u) - \alpha u (1-v)\\
    & \partial_t v = \frac{1}{2} \Delta v - v (1-v)\\
    & u(0,x) = f(x), \quad v(0,x) = g(x).
  \end{cases}
\end{equation}

This non-linear coupling of F-KPP equation was introduced by Holzer \cite{Hol13}. In that article, the author conjectured this partial differential equation to exhibit an anomalous spreading phenomenon, and conjectured a phase diagram for the model \cite[Figure 1]{Hol13}. Our main results confirm this conjecture, and the diagram we obtain in Figure~\ref{figure} exactly matches (up to an adaptation of the notation $\sigma^2 \rightsquigarrow 2d$, $\beta \rightsquigarrow \alpha$ and $\alpha  \rightsquigarrow \beta$) the one obtained by Holzer. Additionally, Theorems~\ref{thm:mainI}--\ref{thm:mainIII} give the position of the front of $v_t$ in \eqref{eqn:fkppmulti}.

When starting with well chosen initial conditions $f$,$g$ (for example such that there exists $A > 0$ satisfying $f(x) =g(x) = 1$ for $x < -A$ and $f(x) = g(x) = 0$ for $x  > A$), we obtain the existence of a function $v_t$ such that for all $x \in \R$,
\[
  \lim_{t \to \infty} (u(t,x - v_t),v(t,x- v_t)) = (w_1(x),w_2(x)),
\]
where $(w_1,w_2)$ is a travelling wave solution of the coupled PDE and:
\begin{enumerate}
  \item if $(\beta,\sigma^2) \in \mathcal{C}_I$, then $v_t = \sqrt{2 \beta \sigma^2} t - \frac{3}{2\sqrt{2\beta/\sigma^2}} \log t$;
  \item if $(\beta,\sigma^2) \in \mathcal{C}_{II}$, then $v_t = \sqrt{2} t - \frac{3}{2\sqrt{2}} \log t$;
  \item if $(\beta,\sigma^2) \in \mathcal{C}_{III}$, then $v_t = v t$, with $v$ defined in Theorem~\ref{thm:mainIII}.
\end{enumerate}

Holzer further studied a linearised version of \eqref{eqn:fkppmulti} in \cite{Hol16}, and showed the presence of an anomalous spreading property in that context. However, the phase diagram in that case is of a different nature as the one we obtain in Figure~\ref{figure}. We believe that the phase diagram of this linearised PDE equation should be related to first moment estimates on the number of particles above a given level in the multitype BBM.

The equation \eqref{eqn:fkppmulti} should also be compared to the partial differential equation studied in \cite{BoC14}. In that article, they considered a population with a family of traits indexed by a parameter $\theta \in (\theta_{\min},\theta_{\max})$, that modifies the motility of particles. This was proposed as a model for the invasion of cane toads is Australia, as that population consists of faster individuals, that sacrifice part of their reproduction power as a trade off, and slower individuals that reproduce more easily. The multitype BBM we consider here could then be thought of as some toy-model for this partial differential equation.

\subsection{Future developments}
\label{subsec:oq}

We recall that Theorems \ref{thm:mainI}--\ref{thm:mainIII} cover the asymptotic behaviour of the two-type reducible BBM assuming that $(\beta,\sigma^2) \in \mathcal{C}_I \cup \mathcal{C}_{II} \cup \mathcal{C}_{III}$. However, it does not give the asymptotic behaviour of this process when $(\beta,\sigma^2)$ belongs to the boundary of this set. Understanding the behaviour of the process at these points could help understanding the phase transitions occurring between the different areas of the state space. This would allow results similar to the ones developed in \cite{BoH20} for time-inhomogeneous BBM to be considered in reducible multitype BBM.

We conjecture the following behaviours for the branching Brownian motion at the boundary between areas $\mathcal{C}_I$ and $\mathcal{C}_{III}$.
\begin{conjecture}
Assume that $\beta > 1$ and $\sigma^2 = \frac{\beta}{2\beta - 1}$, then there exist $c > 0$ and $\tilde{\mathfrak{D}}$ such that
\[
  \sum_{u \in \mathcal{N}_t^2} \delta_{X_u(t) - \sqrt{2\beta\sigma^2}t + \frac{1}{\sqrt{2\beta/\sigma^2}} \log t}
\]
converges to a DPPP($c Z^{(1)}_\infty e^{-\sqrt{2\beta/\sigma^2} x}\dd x, \tilde{\mathfrak{D}}$).
\end{conjecture}
Indeed, in this situation, particles $u$ of type~$2$ contributing to the extremal process are expected to satisfy $t - T(u) = O(t^{1/2})$. Therefore, the extremal process keeps an intensity driven by the derivative martingale of particles of type~$1$, and the decoration point measure is given by the extremal process of a BBM of particles of type~$2$ conditioned to travel at speed $\sqrt{2 \beta \sigma^2} > \sqrt{2}$.

Similarly, at the boundary between areas $\mathcal{C}_{II}$ and $\mathcal{C}_{III}$, the following behaviour is expected.
\begin{conjecture}
Assume that $\beta > 1$ and $\sigma^2 = 2 -\beta$, then there exist $c > 0$ and a random variable $\tilde{Z}$ such that
\[
  \sum_{u \in \mathcal{N}_t^2} \delta_{X_u(t) - \sqrt{2}t + \frac{1}{\sqrt{2}} \log t}
\]
converges to a DPPP($c \tilde{Z} e^{-\sqrt{2} x}\dd x, {\mathfrak{D}}$).
\end{conjecture}
There, we used the fact that particles $u$ of type~$2$ contributing to the extremal process are expected to satisfy $T(u) = O(t^{1/2})$.

In the case when $\beta < 1$ and $\sigma^2 \beta = 1$, which corresponds to the boundary between cases $\mathcal{C}_I$ and $\mathcal{C}_{II}$, the picture is less clear as at all time $s$ between $0$ and $t$, particles should have the same probability to reach the maximal position, at least to the first order, as the BBM of particles of type~$1$ and of particles of type~$2$ have same speed.

Further generalisations of the model we consider in this article could be considered. A more general reducible multitype branching Brownian motions with a finite number of states would be expected to exhibit a similar behaviour. One could also allow particles to have different drift coefficients in addition to the different variance terms and branching rates. In that situation, one expects an optimization problem similar to the one studied in \cite{Mal15a} to appear, with a similar resolution of proving that the trajectory followed by particles reaching the maximal position is the same as the one inferred from the solution of the optimization problem.

Proving Theorems \ref{thm:mainI}--\ref{thm:mainIII} for two-type reducible branching Brownian motions in which particles of type~$1$ and type~$2$ split into a random number of children at each branching event, say $L_1$ for particles of type~$1$ and $L_2$ for particles of type~$2$ would be a other natural generalisation of our results. A natural condition to put on the reproduction laws to obtain the asymptotic behaviour observed in Theorem~\ref{thm:mainIII} is
\[
  \E (L_1 \log L_1 ) + \E (L_2 \log L_2) < \infty.
\]
It is worth noting that anomalous spreading might occur even if $\E(L_2) < 1$, i.e. even if the genealogical tree of a particle of type~$2$ is subcritical and grows extinct almost surely.

While we only take interest here in the asymptotic behaviour of the extremal particles in this article, we believe that many other features of multitype branching Brownian motions might be of interest, such as the growth rate of the number of particles of type~$2$ to the right of $a t$ for $a < v$, the large deviations of the maximal displacement $M_t$ at time $t$, or the convergence of associated (sub)-martingales.

\section{Preliminary results on the branching Brownian motion}
\label{sec:facts}

We list in this section results on the standard BBM, that we use to study the two-type reducible BBM. For the rest of the section, $(X_u(t), u \in \mathcal{N}_t)_{t \geq 0}$ will denote a standard BBM, with branching rate $1$ and diffusion constant $1$, i.e. that has the same behaviour as particles of type $2$. To translate the results of this section to the behaviour of particles of type $1$ as well, it is worth noting that for all $\beta, \sigma > 0$:
\begin{equation}
  \label{eqn:scaling}
  \left(\frac{\sigma}{\sqrt{\beta}} X_u(\beta t), u \in \mathcal{N}_{\beta t}\right)_{t \geq 0}
\end{equation}
is a branching Brownian with branching rate $\beta$ and diffusion constant $\sigma^2$.

The rest of the section is organised as follows. We introduce in Section~\ref{subsec:martingales} the additive martingales of the BBM, and in particular the derivative martingale that plays a special role in the asymptotic behaviour of the maximal displacement of the BBM. We then provide in Section~\ref{subsec:maxdisp} a series of uniform asymptotic estimates on the maximal displacement of the BBM. Finally, in Section~\ref{subsec:exProc}, we introduce the decoration measures and extremal processes appearing when studying particles near the rightmost one in the BBM.

\subsection{Additive martingales of the branching Brownian motion}
\label{subsec:martingales}

We begin by introducing the additive martingales of the BBM. For all $\theta \in \R$, the process
\begin{equation}
  \label{eqn:defW}
  W_t(\theta) := \sum_{u \in \mathcal{N}_t} e^{\theta X_u(t) - t \left(\tfrac{\theta^2}{2}+1\right)}, \quad t \geq 0
\end{equation}
is a non-negative martingale. It is now a well-known fact that the martingale $(W_t(\theta), t \geq 0)$ is uniformly integrable if and only if $|\theta| < \sqrt{2}$, and in that case it converges towards an a.s. positive limiting random variable 
\begin{equation}
  \label{eqn:limitW}
  W_\infty(\theta) := \lim_{t \to \infty} W_t(\theta).
\end{equation}
Otherwise, we have $\lim_{t \to \infty} W_t(\theta) = 0$ a.s. This result was first shown by \cite{Nev}. It can also be obtained by a specific change of measure technique, called \emph{the spinal decomposition}. This method was pioneered by Lyons, Pemantle and Peres \cite{LPP95} for the study of the martingale of a Galton-Watson process, and extended by Lyons \cite{Lyo97} to spatial branching processes setting.

For all $|\theta| < \sqrt{2}$ the martingale limit $W_\infty(\theta)$ is closely related to the number of particles moving at speed $\theta$ in the BBM. For example, by \cite[Corollary 4]{Big92}, for all $h > 0$ we have
\begin{equation}
  \label{eqn:nbPartsSpeedtheta}
  \lim_{t \to \infty} t^{1/2} e^{t \left(\frac{\theta^2}{2} - 1 \right)}\sum_{u \in \mathcal{N}_t} \ind{|X_u(t) - \theta t| \leq h} = \frac{2 \sinh(\theta h)}{\theta} W_\infty(\theta) \quad \text{a.s.}
\end{equation}
This can be thought of as a local limit theorem result for the position of a particle sampled at random at time $t$, where a particle at position $x$ is sampled with probability proportional to $e^{\theta x}$. A Donsker-type theorem was obtained in \cite[Section C]{Pai} for this quantity, see also \cite{GKS18}. In particular, for any continuous bounded function $f$, one has
\begin{equation}
  \label{eqn:cltSpeedTheta}
  \lim_{t \to \infty} \sum_{u \in \mathcal{N}_t} f\left(\tfrac{X_u(t) - \theta t}{t^{1/2}}\right) e^{\theta X_u(t) - t \left(\frac{\theta^2}{2} + 1 \right)} =  \frac{W_\infty(\theta)}{\sqrt{2\pi}} \int_\R e^{-\frac{z^2}{2}} f(z) \dd z \text{ a.s.}
\end{equation}
This justifies the fact that the variable $W_\infty(\theta)$ appears in the limiting distribution of the extremal process in the anomalous spreading case, by the heuristics described in Section~\ref{subsec:stateoftheart}.

To prove Theorem \ref{thm:mainIII}, we use the following slight generalization of the above convergence.
\begin{lemma}
\label{lem:cltExpanded}
Let $a < b$ and $\lambda > 0$. For all continuous bounded function $f : [a,b] \times \R \to \R$, we have
\begin{equation}
  \label{eqn:cltExpanded}
  \lim_{t \to \infty} \frac{1}{t^{1/2}}\int_{\lambda t + a t^{1/2}}^{\lambda t + bt^{1/2}} \sum_{u \in \mathcal{N}_s} f\left(\tfrac{s - \lambda t}{t^{1/2}},\tfrac{X_u(s) - \theta s}{t^{1/2}}\right) e^{\theta X_u(s) - s \left(\frac{\theta^2}{2} + 1 \right)} \dd s =  \frac{W_\infty(\theta)}{\sqrt{2\pi\lambda}} \int_{[a,b] \times \R} e^{-\frac{z^2}{2\lambda}} f(r,z) \dd r \dd z \text{ a.s.}
\end{equation}
\end{lemma}

\begin{proof}
As a first step, we show that \eqref{eqn:cltExpanded} holds for $f : (r,x) \mapsto \ind{r \in [a,b]} g(x)$, with $g$ a continuous compactly supported function. Using that
\[
  \lim_{s \to \infty} \sum_{u \in \mathcal{N}_s} g\left(\lambda^{1/2}\tfrac{X_u(s) - \theta s}{s^{1/2}}\right) e^{\theta X_u(s) - s \left(\frac{\theta^2}{2} + 1 \right)} =  \frac{W_\infty(\theta)}{\sqrt{2\pi}} \int_\R e^{-\frac{z^2}{2}} g(\lambda^{1/2}z) \dd z \quad \text{a.s,}
\]
by \eqref{eqn:cltSpeedTheta} we immediately obtain that
\[
  \lim_{t \to \infty} \frac{1}{t^{1/2}}\int_{\lambda t + a t^{1/2}}^{\lambda t + bt^{1/2}} \sum_{u \in \mathcal{N}_s} g\left(\lambda^{1/2}\tfrac{X_u(s) - \theta s}{s^{1/2}}\right) e^{\theta X_u(s) - s \left(\frac{\theta^2}{2} + 1 \right)} \dd s =  (b-a)\frac{W_\infty(\theta)}{\sqrt{2\pi\lambda}} \int_{\R} e^{-\frac{z^2}{2\lambda}} g(z) \dd z \text{ a.s.}
\]

We then observe that for all $s \in [\lambda t + at^{1/2},\lambda t + bt^{1/2}]$, we have
\[
  s^{1/2} = \left(\lambda t + s - \lambda t \right)^{1/2} = (\lambda t)^{1/2} \left( 1 + \frac{s - \lambda t}{\lambda t} \right)^{1/2}.
\]
As $\frac{s - \lambda t}{\lambda t} \in [at^{-1/2}/\lambda, b t^{-1/2}/\lambda]$, there exists a constant $K> 0$ such that for all $t \geq 1$, we have
\[
  \sup_{s \in [\lambda t + at^{1/2},\lambda t + bt^{1/2}]} \left| \left( 1 + \frac{s - \lambda t}{\lambda t} \right)^{1/2} - 1 \right| \leq K t^{-1/2},
\]
so $\left| s^{1/2} - (\lambda t)^{1/2} \right| \leq K \lambda^{1/2}$ uniformly in $s \in [\lambda t + a t^{1/2},\lambda t + b t^{1/2}]$, for all $t$ large enough.

Then, using the uniform continuity and compactness of $g$, for all $\epsilon > 0$ we have
\[
  \sup_{s \in [\lambda t + a t^{1/2},\lambda t + b t^{1/2}b], x \in \R} \left|g(\lambda^{1/2}x/s^{1/2}) - g(x / t^{1/2}) \right| \leq \epsilon
\]
for all $t$ large enough. Therefore, for all $\epsilon > 0$,
\begin{multline*}
  \limsup_{t \to \infty} \frac{1}{t^{1/2}}\int_{\lambda t + a t^{1/2}}^{\lambda t + bt^{1/2}} \sum_{u \in \mathcal{N}_s} \left|g\left(\tfrac{X_u(s) - \theta s}{s^{1/2}}\right) - g\left(\tfrac{X_u(s) - \theta s}{(\lambda t)^{1/2}}\right)\right| e^{\theta X_u(s) - s \left(\frac{\theta^2}{2} + 1 \right)} \dd s \\
  \leq  \limsup_{t \to \infty} \frac{1}{t^{1/2}} \int_{\lambda t + a t^{1/2}}^{\lambda t + bt^{1/2}} \sum_{u \in \mathcal{N}_s} \epsilon e^{\theta X_u(s) - s \left(\frac{\theta^2}{2} + 1 \right)} \dd s  = \epsilon (b-a) W_\infty(\theta) \quad \text{a.s.}
\end{multline*}
Letting $\epsilon \to 0$, we finally obtain
\begin{equation}
  \label{eqn:step1clt}
  \lim_{t \to \infty} \frac{1}{t^{1/2}}\int_{\lambda t + a t^{1/2}}^{\lambda t + bt^{1/2}} \sum_{u \in \mathcal{N}_s} f\left(\tfrac{s - \lambda t}{t^{1/2}},\tfrac{X_u(s) - \theta ^s}{t^{1/2}}\right) e^{\theta X_u(s) - s \left(\frac{\theta^2}{2} + 1 \right)} \dd s =  \frac{W_\infty(\theta)}{\sqrt{2\pi\lambda}} \int_{[a,b] \times \R} \!\!\! e^{-\frac{z^2}{2\lambda}} f(r,z) \dd r \dd z \quad \text{a.s.}
\end{equation}

We now assume that $f$ is a continuous compactly supported function on $[a,b] \times \R$. For all $i \leq n$, we set 
\[f_i(r,x) = \ind{r \in [a + i(b-a)/n,a + (i+1)(b-a)/n]} f(a + i(b-a)/n,x).\]
Using the uniform integrability of $f$, for all $n$ large enough, we have $\norme{f - \sum_{j=1}^n f_j}_\infty \leq \epsilon$. As a result, we have
\[
  \limsup_{t \to \infty} \frac{1}{t^{1/2}}\int_{\lambda t + a t^{1/2}}^{\lambda t + bt^{1/2}} \sum_{u \in \mathcal{N}_s} \left|f\left(\tfrac{s - \lambda t}{t^{1/2}},\tfrac{X_u(s) - \theta s}{t^{1/2}}\right) - \sum_{i = 1}^n f_i\left(\tfrac{s - \lambda t}{t^{1/2}},\tfrac{X_u(s) - \theta s}{t^{1/2}}\right)\right| e^{\theta X_u(s) - s \left(\frac{\theta^2}{2} + 1 \right)} \dd s \leq \epsilon W_\infty(\theta) \text{ a.s.}
\]
Therefore, using \eqref{eqn:step1clt}, we obtain
\begin{multline*}  \limsup_{t \to \infty} \Bigg|\frac{1}{t^{1/2}}\int_{\lambda t + a t^{1/2}}^{\lambda t + bt^{1/2}} \sum_{u \in \mathcal{N}_s} f\left(\tfrac{s - \lambda t}{t^{1/2}},\tfrac{X_u(s) - \theta s}{t^{1/2}}\right) e^{\theta X_u(s) - s \left(\frac{\theta^2}{2} + 1 \right)} \dd s\\ - \frac{W_\infty(\theta)}{\sqrt{2\pi\lambda}}\int_{[a,b] \times \R}\!\!\! e^{-\frac{z^2}{2\lambda}} f(r,z) \dd r \dd z \Bigg|
\leq 2 \epsilon W_\infty(\theta) \quad \text{a.s.}
\end{multline*}
Letting $\epsilon \to 0$ therefore proves that \eqref{eqn:cltExpanded} holds for compactly supported continuous functions.

Finally, to complete the proof, we consider a continuous bounded function $f$ on $[a,b] \times \R$. Let $R > 0$, given $\chi_R$ a continuous function on $\R$ such that $\ind{|x|<R} \leq \chi_R(x) \leq \ind{|x| \leq R+1}$, the previous computation shows that
\begin{multline*}
  \lim_{t \to \infty} \frac{1}{t^{1/2}}\int_{\lambda t + a t^{1/2}}^{\lambda t + bt^{1/2}} \sum_{u \in \mathcal{N}_s} \chi_R\left(\tfrac{X_u(s) - \theta s}{t^{1/2}}\right)f\left(\tfrac{s - \lambda t}{t^{1/2}},\tfrac{X_u(s) - \theta s}{t^{1/2}}\right) e^{\theta X_u(s) - s \left(\frac{\theta^2}{2} + 1 \right)} \dd s\\
  =  \frac{W_\infty(\theta)}{\sqrt{2\pi\lambda}} \int_{[a,b] \times \R} e^{-\frac{z^2}{2\lambda}} f(r,z) \chi_R(z) \dd r \dd z \text{ a.s.}
\end{multline*}
Additionally, setting $K = \norme{f}_\infty$, for all $t$ large enough we have
\begin{multline*} 
  \left|\frac{1}{t^{1/2}}\int_{\lambda t + a t^{1/2}}^{\lambda t + bt^{1/2}} \sum_{u \in \mathcal{N}_s} \left( 1 - \chi_R\left(\tfrac{X_u(s) - \theta s}{t^{1/2}}\right)\right)f\left(\tfrac{s - \lambda t}{t^{1/2}},\tfrac{X_u(s) - \theta s}{t^{1/2}}\right) e^{\theta X_u(s) - s \left(\frac{\theta^2}{2} + 1 \right)} \dd s\right|\\
  \leq \frac{K}{t^{1/2}}\int_{\lambda t + a t^{1/2}}^{\lambda t + bt^{1/2}} \sum_{u \in \mathcal{N}_s} \left( 1 - \chi_R\left(\lambda^{1/2}\tfrac{X_u(s) - \theta s}{2 s^{1/2}}\right)\right) e^{\theta X_u(s) - s \left(\frac{\theta^2}{2} + 1 \right)}\dd s,
\end{multline*}
which converges to $\frac{(b-a)W_\infty(\theta)}{\sqrt{2\pi}} \int_\R (1 - \chi_R(\lambda^{1/2}z/2)) e^{-\frac{z^2}{2}} \dd z$ as $t \to \infty$. Thus, letting $t \to \infty$ then $R \to \infty$ completes the proof of this lemma.
\end{proof}

\subsection{The derivative martingale}

The number of particles that travel at the critical speed~$\sqrt{2}$ cannot be counted using the additive martingale (as it converges to $0$ almost surely). In this situation, the appropriate process allowing this estimation is the derivative martingale $(Z_t, t \geq 0)$. Its name comes from the fact that $Z_t$ can be represented as $\left.-\frac{\partial }{\partial \theta} W_t(\theta)\right|_{\theta= \sqrt{2}}$, more precisely
\begin{equation}
  Z_t := \sum_{u \in \mathcal{N}_t} (\sqrt{2} t - X_u(t)) e^{\sqrt{2}X_u(t) - 2t}.
\end{equation}
Despite being a non-integrable signed martingale, it was proved by Lalley and Sellke \cite{LaS87} that it converges to an a.s. positive random variable
\begin{equation}
  \label{eqn:cvDerivative}
  Z_\infty := \lim_{t \to \infty} Z_t \quad \text{ a.s.}
\end{equation}
In the same way that the limit of the additive martingale gives the growth rate of the number of particles moving at speed $\theta$, the derivative martingale gives the growth rate of particles that go at speed $\sqrt{2}$. As a result, it appears in the asymptotic behaviour of the maximal displacement, and results similar to \eqref{eqn:nbPartsSpeedtheta} and \eqref{eqn:cltSpeedTheta} can be found in \cite{Mad,Pai} in the context of branching random walks.

We mention that the limit $Z_\infty$ of the derivative martingale is non-integrable, and that its precise tail has been well-studied. In particular, Bereskycki, Berestycki and Schweinsberg \cite{BBS13} proved that
\begin{equation}
  \label{eqn:tailDerivativeMartingale}
  \P(Z_\infty \geq x) \sim \frac{\sqrt{2}}{x} \text{ as } x \to \infty.
\end{equation}
Similar results were obtained for branching random walks by Buraczewski \cite{Bur09} and Madaule \cite{Mad}. They also obtained a more precise estimate on its asymptotic, that can be expressed in the two following equivalent ways
\begin{align}
  \label{eqn:asymptoticsDerivativeMartingale}
  \E(Z_\infty \ind{Z_\infty \leq x}) = \sqrt{2}\log x + O(1) \quad \text{ as } x \to \infty, \\
  1 - \E(e^{-\lambda Z_\infty}) = \sqrt{2}\lambda \log \lambda + O(\lambda) \quad \text{ as } \lambda \to 0.
  \label{eqn:asymptoticsDerivativeMartingale2}
\end{align}
Maillard and Pain \cite{MaP19} improved on these statements and gave necessary and sufficient conditions for the asymptotic developments of these quantities up to a $o(1)$. We mention that the equivalence between \eqref{eqn:asymptoticsDerivativeMartingale} and \eqref{eqn:asymptoticsDerivativeMartingale2} can be found in \cite[Lemma~8.1]{BIM20+}, which obtain similar necessary and sufficient conditions for the asymptotic development of the Laplace transform of the derivative martingale of the branching random walk under optimal integrability conditions.

\subsection{Maximal displacement of the branching Brownian motion}
\label{subsec:maxdisp}

A large body of work has been dedicated to the study of the maximal displacement of the BBM, defined by $M_t = \max_{u \in \mathcal{N}_t}X_u(t)$. We recall here some estimates related to its study. We begin by observing that the BBM travels in a triangular-shaped array, and that for all $y \geq 0$
\begin{equation}
  \label{eqn:triangleShape}
  \P\left( \exists t \geq 0, u \in \mathcal{N}_t : X_u(t) \geq \sqrt{2} t + y\right) \leq e^{-\sqrt{2} y},
\end{equation}
which shows that with high probability, all particles at time $t$ are smaller than $\sqrt{2}t + y$ in absolute value.

Recall that Lalley and Sellke \cite{LaS87} proved that setting $ m_t = \sqrt{2} t - \frac{3}{2\sqrt{2}} \log t$, the maximal displacement of the BBM centered by $m_t$ converges in distribution to a shifted Gumbel distribution. More precisely, there exists $c_\star > 0$ such that
\begin{equation}
  \label{eqn:cvInLawMAx}
  \lim_{t \to \infty} \P(M_t \leq m_t + z)  = \E\left( \exp\left( - c_\star Z_\infty e^{-\sqrt{2} z} \right) \right).
\end{equation}
An uniform upper bound is also known for the right tail of the maximal displacement. There exists $C > 0$ such that for all $t \geq 0$ and $x \in \R$, we have
\begin{equation}
  \label{eqn:asymptoticsMaximaldisplacement}
  \P(M_t \geq m_t + x) \leq C (1 + x_+) e^{-\sqrt{2} x},
\end{equation}
where $x_+ = \max(x,0)$. This estimate can be obtained by first moment methods, we refer e.g. to \cite{Hu16} for a similar estimate in the branching random walk, which immediately implies a similar bound for the BBM.

In the context of the anomalous spreading, seen from the heuristics in Section \ref{subsec:stateoftheart}, it will also be necessary to use tight estimates on the large deviations of the BBM. These large deviations were first studied by Chauvin and Rouault \cite{ChR88}. Precise large deviations for the maximal displacement were recently obtained in \cite{DMS16,GaH18, BuM19,BBCM19}, proving that for all $\rho > \sqrt{2}$, there exists $C(\rho) \in (0,1)$ such that
\begin{equation}
  \label{eqn:largeDevMax}
  \P(M_t \geq \rho t + y) \sim_{t \to \infty}  \frac{C(\rho)}{\sqrt{2 \pi t} \rho} e^{-(\rho^2/2-1)t} e^{-\rho y - \tfrac{y^2}{2t}},
\end{equation}
uniformly in $|y| \leq r_t$, for all function $r_t = o(t)$.

Additionally, from a simple first moment estimate, one can obtain an uniform upper bound for this large deviations estimate on the maximal displacement.
\begin{lemma}
\label{lem:unifUBLargeDeviationsMax}
For all $\rho > \sqrt{2}$ and $A > 0$, there exists $C > 0$ such that for all $t$ large enough and all $y \geq - A t^{1/2}$, we have
\[
  \P(M_t \geq \rho t + y) \leq \frac{Ce^{- (\rho^2/2-1)t}}{t^{1/2}} e^{-\rho y - \tfrac{y^2}{2t}}. 
\]
\end{lemma}

This result is based on Markov inequality and classical Gaussian estimates, that appear later in our paper in more complicate settings. We thus give a short proof of this statement.
\begin{proof}
Observe that for $t$ large enough, we have $\rho t + y \geq \delta t$ for some positive constant $\delta$. Then, by Markov inequality, we have
\begin{equation*}
  \P(M_t \geq \rho t +y) = \P\left( \exists u \in \mathcal{N}_t : X_u(t) \geq \rho t + y\right)\\
  \leq \E\left( \sum_{u \in \mathcal{N}_t} \ind{X_u(t) \geq \rho t + y} \right).
\end{equation*}
Using that there are on average $e^{t}$ particles alive at time $t$ and that the displacements of particles are Brownian motions, that are independent of the total number of particles in the process, we have
\[
  \E\left( \sum_{u \in \mathcal{N}_t} \ind{X_u(t) \geq \rho t + y} \right) = e^t \P(B_t \geq \rho t + y).
\]
This fact is often called the many-to-one lemma in the literature (see e.g. \cite[Theorem 1]{ShiBook}). We develop in Section \ref{sec:mto} a multitype versions of this result.

We now use the following well-known asymptotic estimate on the tail of the Gaussian random variable that
\begin{equation}
  \label{eqn:gaussianEstimate}
  \P(B_1 \geq x) \leq \frac{1}{\sqrt{2 \pi} x}e^{-\tfrac{x^2}{2}} \quad \text{ for all $x \geq 0$}.
\end{equation}
This yields
\begin{align*}
  \E\left( \sum_{u \in \mathcal{N}_t} \ind{X_u(t) \geq \rho t + y} \right)&=  e^t \P\left(B_1 \geq \rho t^{1/2} + yt^{-1/2}\right) \leq C t^{-1/2} e^{- \tfrac{(\rho t+y)^2}{2t}}\\
  &\leq C t^{-1/2} e^{t(1 - \rho^2/2)} e^{-\rho y - \tfrac{y^2}{2t}}
\end{align*}
completing the proof.
\end{proof}

\subsection{Decorations of the branching Brownian motion}
\label{subsec:exProc}

We now turn to results related to the extremal process of the BBM. Before stating these, we introduce a general tool that allows the obtention of the joint convergence in distribution of the maximal displacement and the extremal process of a particle system. Denote by $\mathcal{T}$ the set of continuous non-negative bounded functions, with support bounded on the left. The following result can be found in \cite[Lemma~4.4]{BBCM19}.
\begin{proposition}
\label{prop:convergencePointProcess}
Let $\mathcal{P}_n, \mathcal{P}$ be point measures on the real line. We denote by $\max \mathcal{P}_n$ (respectively $\max \mathcal{P}$), the position of the rightmost atom in this point measure. The following statements are equivalent
\begin{enumerate}
  \item $\lim_{n \to \infty} \mathcal{P}_n = \mathcal{P}$ and $\lim_{n \to \infty} \max \mathcal{P}_n = \max \mathcal{P}$ in law.
  \item $\lim_{n \to \infty} (\mathcal{P}_n, \max \mathcal{P}_n) = (\mathcal{P},\max \mathcal{P})$ in law.
  \item for all $\phi \in \mathcal{T}$, $\lim_{n \to \infty} \E\left( e^{-\crochet{\mathcal{P}_n,\phi}} \right) = \E\left( e^{-\crochet{\mathcal{P},\phi}} \right)$.
\end{enumerate}
\end{proposition}

In other words, considering continuous bounded functions with support bounded on the left instead of continuous compactly supported functions allow us to capture the joint convergence in law of the maximal displacement and the extremal process. We refer to the set $\mathcal{T}$ as the set of test functions, against which we test the convergence of our point measures of interest.

The convergence in distribution of the extremal process of a BBM has been obtained by Aïdékon, Berestycki, Brunet and Shi \cite{ABBS}, and by Arguin, Bovier and Kistler \cite{ABK3}. They proved that setting
\[
  \mathcal{E}_t = \sum_{u \in \mathcal{N}_t} \delta_{X_u(t)-m_t},
\]
this extremal process converges in distribution towards a decorated Poisson point process with intensity $c_\star Z_\infty \sqrt{2} e^{-\sqrt{2} z} \dd z$. The law of the decoration is described in \cite{ABK3} as the limiting distribution of the maximal displacement seen from the rightmost particle, conditioned on being larger than $\sqrt{2} t$ at time $t$. More precisely, they proved that there exists a point measure $\mathcal{D}$ such that
\begin{equation}
  \label{eqn:defineDecoration}
  \lim_{t \to \infty} \E\left(\exp\left( - \sum_{u \in \mathcal{N}_t} \phi(X_u(t) - M_t) \right) \middle|M_t \geq \sqrt{2} t \right) = \E\left( \exp\left( - \crochet{\mathcal{D},\phi} \right) \right)
\end{equation}
for all function $\phi \in \mathcal{T}$. Note that $\mathcal{D}$ is supported on $(-\infty, 0]$ and has an atom at $0$.

The limiting extremal process $\mathcal{E}_\infty$ can be constructed as follows. Let $(\xi_j)_{j \in \N}$ be the atoms of a Poisson point process with intensity $c_\star \sqrt{2} e^{-\sqrt{2} z} \dd z$, and $(\mathcal{D}_j, j \in \N)$ i.i.d. point measures, then set
\[
  \mathcal{E}_\infty = \sum_{j \in \N} \sum_{d \in \mathcal{D}_j} \delta_{\xi_j + d + \frac{1}{\sqrt{2}} \log Z_\infty},
\]
where $\sum_{d \in \mathcal{D}_j}$ represents a sum on the set of atoms of the point measure $\mathcal{D}_j$.

In view of Proposition \ref{prop:convergencePointProcess} and \eqref{eqn:cvInLawMAx}, we can rewrite as follows the convergence in law of the extremal process of the BBM, with simple Poisson computations.
\begin{lemma}
\label{lem:abbs}
For all function $\phi \in \mathcal{T}$, we have
\[
  \lim_{t \to \infty} \E\left( e^{-\crochet{\mathcal{E}_t,\phi}} \right)
  =\E\left( \exp\left(-c_\star  Z_\infty \int (1-e^{-\Psi[\phi](z)}) \sqrt{2} e^{-\sqrt{2} z} \dd z\right) \right),
\]
where we have set $\Psi[\phi] : z \mapsto -\log \E \left( e^{-\crochet{\mathcal{D},\phi(\cdot+z)}} \right)$.
\end{lemma}

In the context of large deviations of BBM, a one-parameter family of point measures, similar to the one defined in \eqref{eqn:defineDecoration} can be introduced. These point measures have first been studied by Bovier and Hartung \cite{BoH14} when considering the extremal process of the time-inhomogeneous BBM. More precisely, they proved that for all $\rho > \sqrt{2}$, there exists a point measure $\mathcal{D}^\rho$ such that
\begin{equation}
  \label{eqn:defineSupercriticalDecoration}
  \lim_{t \to \infty} \E\left(\exp\left( - \sum_{u \in \mathcal{N}_t} \phi(X_u(t) - M_t) \right) \middle|M_t \geq \rho t \right) = \E\left( \exp\left( - \crochet{\mathcal{D}^\rho,\phi} \right) \right).
\end{equation}

In \cite{BBCM19}, an alternative construction of this one parameter family of point measures was introduced, which allows its representation as a point measure conditioned on an event of positive probability instead of a large deviation event of probability decaying exponentially fast in $t$. Let us begin by introducing a few notation. Let $(B_t, t \geq 0)$ be a standard Brownian motion, $(\tau_k, k \geq 1)$ the atoms of an independent Poisson point process of intensity $2$, and $(X^{(k)}_u(t), u \in \mathcal{N}^{(k)}_t, t \geq 0)$ i.i.d. BBMs, which are further independent of $B$ and $\tau$. For $\rho > \sqrt{2}$, we set
\begin{equation}
  \tilde{\mathcal{D}}_t^\rho = \delta_0 + \sum_{k \geq 1} \ind{\tau_k < t} \sum_{u \in \mathcal{N}^{(k)}_{\tau_k}} \delta_{B_{\tau_k} - \rho \tau_k + X_u(\tau_k)} \quad \text{and} \quad \tilde{\mathcal{D}}^\rho = \lim_{t \to \infty} \tilde{\mathcal{D}}_t^\rho.
\end{equation}
In words, the process $\tilde{\mathcal{D}}^\rho$ consists in making one particle starts from $0$ and travels backwards in time according to a Brownian motion with drift $\rho$. This particle gives birth to offspring at rate $2$, each newborn child starting an independent BBM from its current position, forward in time. The point measure $\tilde{\mathcal{D}}^\rho$ then consists of the position of all particles alive at time $0$.

As a first step, we mention the following result, which can be thought of as a spinal decomposition argument with respect to the rightmost particle. This result can be found in \cite[Lemma 2.1]{BBCM19}.
\begin{proposition}
\label{prop:spine}
For all $t \geq 0$ set
\[
  \mathcal{E}^*_t = \sum_{u \in \mathcal{N}_t} \delta_{X_u(t) - M_t}
\]
the extremal process seen from the rightmost position. For all measurable non-negative functions $f,F$, we have
\[
  \E\left( f(M_t-\rho t) F(\mathcal{E}^*_t)  \right) = e^{(1-\rho^2/2)t} \E\left( e^{-\rho B_t} f(B_t) F(\tilde{\mathcal{D}}_t^\rho) \ind{\tilde{\mathcal{D}}^\rho_t((0,\infty)) = 0} \right)
\]
\end{proposition}

It then follows from \eqref{eqn:largeDevMax} and the above proposition that the law $\mathcal{D}^\rho$ can be represented by conditioning the point measure $\tilde{\mathcal{D}}^\rho$, as was obtained in \cite[Theorem 1.1]{BBCM19}.
\begin{lemma}
\label{lem:representationOfPointMeasures}
For all $\rho > \sqrt{2}$,
\begin{itemize}
  \item the constant $C(\rho)$ introduced in \eqref{eqn:largeDevMax} is given by $C(\rho) = \P(\tilde{\mathcal{D}}^\rho((0,\infty))=0)$.
  \item the law of the point measure $\mathcal{D}^\rho$ introduced in \eqref{eqn:defineSupercriticalDecoration} can be constructed as
  \[
    \P(\mathcal{D}^\rho \in \cdot) = \P(\tilde{\mathcal{D}}^\rho \in \cdot | \tilde{\mathcal{D}}^\rho((0,\infty))=0 ).
  \]
\end{itemize}
\end{lemma}

We end this section with an uniform estimate on the Laplace transform of the extremal process of the BBM, that generalizes both \eqref{eqn:largeDevMax} and \eqref{eqn:defineSupercriticalDecoration}.
\begin{lemma}
\label{lem:jointLargeDev}
Let $\rho > \sqrt{2}$, we set
\[
  \mathcal{E}_t^\rho(x) = \sum_{u \in \mathcal{N}_t} \delta_{X_u(t) - \rho t + x}.
\]
Let $A > 0$, for all $\phi \in \mathcal{T}$, we have
\[
  \E\left( 1- e^{-\crochet{\mathcal{E}_t^\rho(x), \phi}} \right)  = C(\rho)\frac{e^{\left(1 - \rho^2/2\right)t}}{\sqrt{2\pi t}}  e^{\rho x - \tfrac{x^2}{2t}} \int e^{-\rho z} \left(1 - e^{-\Psi^\rho[\phi](z)}\right) \dd z (1 + o(1)),
\]
uniformly in $|x| \leq A t^{1/2}$, as $t \to \infty$, where $\Psi^\rho[\phi] : z \mapsto  -\log \E \left( e^{-\crochet{\mathcal{D}^\rho,\phi(\cdot+z)}} \right)$.
\end{lemma}

\begin{proof}
Let $L > 0$, recall from Lemma \ref{lem:unifUBLargeDeviationsMax} that
\begin{equation*}
  \P(M_t \geq \rho t - x + L) \leq C t^{-1/2} e^{t(1 - \rho^2/2)} e^{\rho x - \tfrac{x^2}{2t}} e^{- \rho L}.
\end{equation*}
Thus, as $\phi$ is non-negative, we have
\begin{equation}
  \label{eqn:ubMax}
  0 \leq \E\left( 1- e^{-\crochet{\mathcal{E}_t^\rho(x), \phi}} \right) -  \E\left( \left(1- e^{-\crochet{\mathcal{E}_t^\rho(x), \phi}}\right) \ind{M_t \leq \rho t - x + L} \right) \leq  C t^{-1/2} e^{t(1 - \rho^2/2)} e^{\rho x - \tfrac{x^2}{2t}} e^{- \rho L}.
\end{equation}
We also recall that the support of $\phi$ is bounded on the left, i.e. is included on $[R,\infty)$ for some $R \in \R$. Observe then that $e^{-\crochet{\mathcal{E}_t^\rho(x), \phi}} = 1$ on the event $\{M_t \leq \rho t - x + R\}$.

We now use Proposition \ref{prop:spine} to compute
\begin{multline*}
  \E\left( \left(1- e^{-\crochet{\mathcal{E}_t^\rho(x), \phi}}\right) \ind{M_t - \rho t + x \in [R,L]} \right)\\
  = e^{t(1 - \rho^2/2)} \E\left( e^{\rho B_t} \left( 1 - e^{-\crochet{\tilde{\mathcal{D}}^\rho_t, \tau_{B_t + x} \phi}} \right)\ind{x + B_t \in [R,L], \tilde{\mathcal{D}}^\rho_t((0,\infty)) = 0} \right),
\end{multline*}
where $\tau_z(\phi)(\cdot) = \phi(z + \cdot)$. Therefore, setting
\[
  G_t(x,z) = \E\left( \left( 1 - e^{-\crochet{\tilde{\mathcal{D}}^\rho_t, \tau_{z} \phi}} \right)\ind{\tilde{\mathcal{D}}^\rho_t((0,\infty)) = 0} \middle| B_t = z-x \right),
\]
we have
\begin{equation}
  \label{eqn:estimate}
  e^{t(\rho^2/2-1)} \sqrt{2\pi t} e^{-\rho x + \tfrac{x^2}{2t}}\E\left( \left(1- e^{-\crochet{\mathcal{E}_t^\rho(x), \phi}}\right) \ind{M_t - \rho t + x \in [R,L]} \right) = \int_{R}^{L} e^{-\rho y + o(t^{-1/2})} G_t(x, y) \dd y,
\end{equation}
with the $o(t^{-1/2})$ term being uniform in $|x| \leq A t^{1/2}$.

With the same computations as in the proof of \cite[Lemma 3.4]{BBCM19}, we obtain
\begin{align*}
  \lim_{t \to \infty} \sup_{|x|\leq A t^{1/2}} G_t(x,y) = \lim_{t \to \infty} \inf_{|x| \leq A t^{1/2}} G_t(x, y) &= \E\left( \left( 1 - e^{-\crochet{\tilde{\mathcal{D}}^\rho, \tau_{y} \phi}} \right)\ind{\tilde{\mathcal{D}}^\rho((0,\infty)) = 0} \right)\\
  &= C(\rho) \E \left( 1 - e^{-\crochet{\mathcal{D}^\rho, \tau_{y} \phi}} \right),
\end{align*}
using the construction of $\mathcal{D}^\rho$ given in Lemma \ref{lem:representationOfPointMeasures}. Therefore, using \eqref{eqn:ubMax} and applying the dominated convergence theorem, equation \eqref{eqn:estimate} yields
\begin{multline*}
  \limsup_{t \to \infty} \sup_{|x| \leq A t^{1/2}} \left| e^{t(\rho^2/2-1)} \sqrt{2\pi t} e^{-\rho x + \tfrac{x^2}{2t}}\E\left( \left(1- e^{-\crochet{\mathcal{E}_t^\rho(x), \phi}}\right) \right) - C(\rho) \int_\R  e^{-\rho y} \E \left( 1 - e^{-\crochet{\mathcal{D}^\rho, \tau_{y} \phi}} \right) \dd y \right| \\
  \leq C e^{- \rho L},
\end{multline*}
which, letting $L \to \infty$, completes the proof.
\end{proof}

\begin{remark}
Note that applying Lemma \ref{lem:jointLargeDev} to function $\phi(z) = \ind{z \geq 0}$ yields \eqref{eqn:largeDevMax}, and up simple computations, this lemma can also be used to obtain \eqref{eqn:defineSupercriticalDecoration}.
\end{remark}

\section{Multitype many-to-one lemmas}
\label{sec:mto}

The many-to-one lemma is an ubiquitous result in the study of branching Brownian motions. This result links additive moments of the BBM with Brownian motion estimates. We first recall the classical version of this lemma, before giving a multitype version that applies to our process.

Let $(X_u(t), u \in \mathcal{N}_t)$ be a standard BBM with branching rate $1$. The classical many-to-one lemma can be tracked back at least to the work of Kahane and Peyrière \cite{KaP,Pey} on multiplicative cascades. It can be expressed as follows: for all $t \geq 0$ and measurable non-negative functions $f$, we have
\begin{equation}
  \label{eqn:many-to-one}
  \E\left( \sum_{u \in \mathcal{N}_t} f(X_u(s), s \leq t) \right) = e^{t} \E(f(B_s, s \leq t)),
\end{equation}
where $B$ is a standard Brownian motion.

Recall that $\mathcal{N}^1_t$ (respectively $\mathcal{N}^2_t$) is the set of particles of type $1$ (resp. type $2$) alive at time $t$. Note that the process $(X_u(t), u \in \mathcal{N}^1_t)_{t \geq 0}$ is a BBM with branching rate $\beta$ and diffusion $\sigma^2$. Thus in view of \eqref{eqn:scaling}, \eqref{eqn:many-to-one} implies that for all measurable non-negative function $f$
\[
  \E\left( \sum_{u \in \mathcal{N}_t^1} f(X_u(s), s \leq t) \right) = e^{\beta t} \E(f(\sigma B_s, s \leq t)).
\]

Similarly, writing $\P^{(2)}$ the law of the process starting from a single particle of type $2$. As this particle behaves as in a standard BBM and only gives birth of particles of type $2$, this process again is a BBM, therefore
\[
  \E^{(2)}\left( \sum_{u \in \mathcal{N}_t^2} f(X_u(s), s \leq t) \right) = e^{t} \E(f(B_s, s \leq t)),
\]
writing $\E^{(2)}$ for the expectation associated to $\P^{(2)}$.

The main aim of this section is to prove the following result, which allows to represent an additive functional of particles of type $2$ appearing in the multitype BBM by a variable speed Brownian motion.
\begin{proposition}
\label{prop:manytoone2type}
For all measurable non-negative function $f$, we have
\[  \E\left( \sum_{u \in \mathcal{N}_t^2} f((X_u(s), s \leq t), T(u)) \right)= \alpha \int_0^t e^{\beta s + (t-s)} \E\left( f( (\sigma B_{u \wedge s} + (B_{u}-B_{u \wedge s}), u \leq t),s) \right) \dd s,
\]
where we recall that $T(u)$ is the birth time of the first ancestor of type $2$ of $u$.
\end{proposition}

To prove this result, we begin by investigating the set $\mathcal{B}$ of particles of type $2$ that are born from a particle of type $1$, that can be defined as
\[
  \mathcal{B} := \left\{ u \in \cup_{t \geq 0} \mathcal{N}^2(t) : T(u) = b_u \right\}.
\]
We observe that $\mathcal{B}$ can be thought of as a Poisson point process with random intensity.
\begin{lemma}
\label{lem:poissonProc}
Conditionally on $\mathcal{F}^1 = \sigma(X_u(t), u \in \mathcal{N}^1_t, t \geq 0)$, the point measure $\displaystyle \sum_{u \in \mathcal{B}} \delta_{(X_u(s), s \leq T(u))}$ is a Poisson point process with intensity
$\displaystyle \alpha \dd t \otimes \sum_{u \in \mathcal{N}^1_t} \delta_{(X_u(s), s \leq t)}$.
\end{lemma}

\begin{proof}
This is a straightforward consequence of the definition of the two-type BBM and the superposition principle for Poisson process. Over its lifetime, a particle of type $1$ gives birth to particles of type $2$ according to a Poisson process with intensity $\alpha$, and the trajectory leading to the newborn particle at time $t$ is exactly the same as the trajectory of its parent particle up to time $t$.
\end{proof}

A direct consequence of the above lemma is the following applications of Poisson summation formula.
\begin{corollary}
\label{cor:poissonSum}
For all measurable non-negative function $f$, we have
\begin{align}
  \E\left(\sum_{u \in \mathcal{B}} f(X_u(s), s \leq T(u)) \right) &= \alpha \int_0^\infty e^{\beta t} \E(f(\sigma B_s, s \leq t)) \dd t,\\
  \E\left( \exp\left( - \sum_{u \in \mathcal{B}} f(X_u(s), s \leq T(u)) \right) \right) &= \E\left( \exp\left( - \alpha \int_0^\infty \sum_{u \in \mathcal{N}_t^1} 1 - e^{-f(X_u(s), s \leq t)} \dd t \right) \right) 
\end{align}
\end{corollary}

\begin{proof}
We denote by $\mathcal{F}^1 = \sigma(X_u(s), u \in \mathcal{N}^1_s, s \geq 0)$ the filtration generated by all particles of type $1$. We can compute
\[
  \E\left(\sum_{u \in \mathcal{B}} f(X_u(s), s \leq T(u)) \middle| \mathcal{F}^1 \right) = \alpha \int_0^\infty \sum_{u \in \mathcal{N}^1_t} f(X_u(s), s \leq t) \dd t
\]
using Lemma \ref{lem:poissonProc}. Then using Fubini's theorem and \eqref{eqn:many-to-one}, we conclude that
\[
\E\left(\sum_{u \in \mathcal{B}} f(X_u(s), s \leq T(u)) \right) = \alpha \int_0^\infty e^{\beta t} \E(f(\sigma B_s, s \leq t)) \dd t.
\]

Similarly, using the exponential Poisson formula, we have
\[
  \E\left(\exp\left(-\sum_{u \in \mathcal{B}} f(X_u(s), s \leq T(u))\right) \middle| \mathcal{F}^1 \right) = \exp\left( - \alpha \int_0^\infty \sum_{u \in \mathcal{N}_t^1} 1 - e^{-f(X_u(s), s \leq t)}\dd t \right).
\]
Taking the expectation of this formula completes the proof of this corollary.
\end{proof}

We now turn to the proof of the multitype many-to-one lemma.
\begin{proof}[Proof of Proposition \ref{prop:manytoone2type}]
Let $f,g$ be two measurable bounded functions. For any $u,u'$ particles in the BBM, we write $u' \succcurlyeq u$ to denote that $u'$ is a descendant of $u$. We compute
\begin{align*}
  &\E\left( \sum_{u \in \mathcal{N}_t^2} f(X_u(s), s \leq T(u)) g(X_u(s),  s \in [T(u),t]) \right)\\
  =  &\E\left( \sum_{u \in \mathcal{B}} \ind{T(u) \leq t} f(X_u(s), s \leq T(u)) \sum_{\substack{u' \in \mathcal{N}^2_t\\u' \succcurlyeq u}}  g(X_{u'}(s), s \in [T(u),t]) \right)\\
  = & \E\left( \sum_{u \in \mathcal{B}} \ind{T(u) \leq t} f(X_u(s), s \leq T(u)) \phi(T(u), X_u(T(u)))\right),
\end{align*}
using the branching property for the BBM: every particle $u \in \mathcal{B}$ starts an independent BBM from time $T(u)$ and position $X_u(T(u))$. Here, we have set for $x \in \R$ and $s \geq 0$
\begin{align*}
  \phi(s,x)
  &= \E^{(2)}\left( \sum_{u \in \mathcal{N}^2_{t-s}} g\left(x + X_u(r - s),r \in [s,t]\right)  \right)\\
  &= e^{t-s} \E\left( g\left(x + B_{r-s}, r \in [s,t]\right) \right),
\end{align*}
by the standard many-to-one lemma. Additionally, by Corollary~\ref{cor:poissonSum}, we have
\begin{align*}
  &\E\left( \sum_{u \in \mathcal{N}_t^2} f(X_u(s), s \leq T(u)) g(X_u(s),  s \in [T(u),t]) \right)\\
  =& \alpha \int_0^t e^{\beta s} \E(f(\sigma B_r, r \leq s) \phi(s,\sigma B_s) ) \dd s\\
  =& \alpha \int_0^t e^{\beta s + t-s} \E(f(\sigma B_r, r \leq s)  g\left(\sigma B_s + (B_{r}-B_s), r \in [s,t])\right) \dd s.
\end{align*}
Using the monotone class theorem, the proof of Proposition \ref{prop:manytoone2type} is now complete.
\end{proof}

\section{Proof of Theorem \ref{thm:mainII}}
\label{sec:proofII}

We assume in this section that $(\beta,\sigma^2) \in \mathcal{C}_{II}$, i.e. that either $\sigma^2 > 1$ and $\sigma^2 < \frac{1}{\beta}$ or $\sigma^2 \leq 1$ and $\sigma^2 < 2 - \beta$. In that case, we show that the extremal process is dominated by the behaviour of particles of type $2$ that are born at the beginning of the process. The main steps of the proof of Theorem \ref{thm:mainII} are the following:
\begin{enumerate}
  \item We show that for all $A > 0$, there exists $R>0$ such that with high probability, every particle $u$ of type~$2$ to the right of $m^{(II)}_t - A$ satisfy $T(u) \leq R$.
  \item We use the convergence in distribution of the extremal process of a single-type branching Brownian motion to demonstrate that the extremal process generated by the individuals born of type $2$ before time $R$ converges as $t \to \infty$.
  \item We prove that letting $R \to \infty$, the above extremal process converges, and the limiting point measure is the point measure of the full two-type branching Brownian motion.
\end{enumerate}

In this section, we write $v = \sqrt{2\beta \sigma^2}$ and $\theta = \sqrt{2\beta/\sigma^2}$, which are respectively the speed and critical parameter of the branching Brownian motion of particles of type $1$. Recall that $m_t^{(II)} = \sqrt{2} t - \frac{3}{2\sqrt{2}} \log t$, we write
\[
  \hat{\mathcal{E}}_t = \sum_{u \in \mathcal{N}^2_t} \delta_{X_u(t) - m^{(II)}_t},
\]
the extremal process of particles of type $2$ in the branching Brownian motion, centred around $m^{(II)}_t$. We begin by proving that with high probability, no particle of type $2$ that was born from a particle of type $1$ after time $R$ has a descendant close to $m_t^{(II)}$.
\begin{lemma}
\label{lem:originII}
Assuming that $(\beta,\sigma^2) \in \mathcal{C}_{II}$, for all $A > 0$, we have
\[
  \lim_{R \to \infty} \limsup_{t \to \infty} \P(\exists u \in \mathcal{N}^2_t : T(u)\geq R, X_u(t)\geq m_t^{(II)}-A) = 0.
\]
\end{lemma}

\begin{proof}
Let $K > 0$, we first recall that by \eqref{eqn:scaling} and \eqref{eqn:triangleShape}, we have
\[
  \P\left( \exists t \geq 0, u \in \mathcal{N}^1_t : X_u(t) \geq v t + K \right) \leq e^{-\theta K},
\]
i.e. that with high probability, all particles of type $1$ stay below the curve $s \mapsto v s + K$.

We now set, for $R,A,K \geq 0$ and $t \geq 0$:
\[
  Y_t(A,R, K) = \sum_{u \in \mathcal{B}} \ind{T(u) > R, X_u(T(u)) \leq v T(u) + K} \ind{M^{u}_t\geq m_t^{(II)}-A},
\]
where $M^u_t$ is the position of the rightmost descendant at time $t$ of the individual $u$. In other words, $Y_t(A,R,K)$ is the number of particles of type $2$ born from a particle of type $1$ after time $R$, that were born below the curve $s \mapsto v s + K$ and have a member of their family to the right of $m_t^{(II)}-A$. Observe that by Markov inequality, we have
\begin{align*}
  \P(\exists u \in \mathcal{N}^2_t : T(u)\geq R, X_u(t)\geq m_t^{(II)}-A) & \leq \P\left( \exists t \geq 0, u \in \mathcal{N}^1_t : X_u(t) \geq v s + K \right) + \P(Y_t(A,R,K) \geq 1) \\
  &\leq e^{-\theta K} + \E(Y_t(A,R,K)). 
\end{align*}
To complete the proof, it is therefore enough to bound $\limsup_{t \to \infty} \E(Y_t(A,R,K))$. Using the branching property and Corollary \ref{cor:poissonSum}, we have
\begin{align}
  \E(Y_t(A,R,K))
  &= \E\left( \sum_{u \in \mathcal{B}} \ind{T(u) \in [R,t]} \ind{X_u(T(u)) \leq v T(u) + K} F\left(t-T(u),X_u(T(u))\right)  \right)\nonumber\\
  &= \alpha \int_R^t e^{\beta s} \E\left( F\left(t-s,\sigma B_s\right) \ind{\sigma B_s \leq v s + K} \right)\dd s, \label{eqn:momentII}
\end{align}
where we have set $F(r,x) = \P^{(2)}\left( x + M_r \geq m_t^{(II)}-A\right)$.

By \eqref{eqn:asymptoticsMaximaldisplacement}, there exists $C > 0$ such that for all $x \in \R$ and $t \geq 0$, we have
\[
  \P^{(2)}\left( M_t \geq m_t^{(II)} + x \right) \leq C (1 + x_+) e^{-\sqrt{2} x},
\]
so that for all $s \leq t$,
\begin{align}
  F(t-s,x) &= \P^{(2)}\left( M_{t-s} \geq m_{t-s}^{(II)} + \sqrt{2} s + \tfrac{3}{2\sqrt{2}} \log \tfrac{t-s+1}{t+1}- A - x \right) \nonumber\\
  &\leq C \left(\frac{t+1}{t-s+1}\right)^\frac{3}{2}\left(1 + \sqrt{2} s + (-x)_+ \right) e^{-\sqrt{2} (\sqrt{2} s - x - A)} \label{eqn:boundPhiII}.
\end{align}
We bound $\E(Y_t(A,R,K))$ in two different ways, depending on the sign of $\sigma^2 - 1$.

First, if $\sigma^2 \leq 1$, we observe that the condition $X_u(s) \leq v s + K$ does not play a major role in the asymptotic behaviour of $\E(Y_t(A,R,K))$. As a result, \eqref{eqn:momentII} and \eqref{eqn:boundPhiII} yield
\[
  \E(Y_t(A,R,K)) \leq C e^A \int_R^t \left(\frac{t+1}{t-s+1}\right)^{3/2} e^{s (\beta - 2)} \E\left( \left( 1 +\sqrt{2}s + \sigma (-B_s)_+ \right) e^{\sqrt{2} \sigma B_s} \right) \dd s,
\]
and as $\E\left(\left( 1 + \sqrt{2 } s + \sigma  (-B_s)_+ \right) e^{\sqrt{2} \sigma B_s} \right) \leq C (1 + s) e^{\sigma^2 s}$, we have
\[  \E( Y_t(A,R,K)) \leq C e^A \int_R^t  \left(\frac{t+1}{t-s+1}\right)^\frac{3}{2}(s +1) \exp\left( s \left( \beta + \sigma^2 - 2 \right) \right) \dd s.
\]
Hence, as $(\beta,\sigma^2) \in \mathcal{C}_{II} $ and $\sigma^2 \leq 1$, we have $\beta + \sigma^2 - 2 < 0$. Therefore, by dominated convergence theorem,
\[
  \limsup_{t \to \infty} \E(Y_t(A,R,K)) \leq C e^A \int_R^\infty (s +1) \exp\left( s \left(\beta + \sigma^2 - 2\right) \right) \dd s,
\]
which goes to $0$ as $R\to \infty$, completing the proof in that case.

We now assume that $\sigma^2 > 1$. In that case, the condition $X_u(s) \leq \sqrt{2 \beta \sigma^2} s + K$ is needed to keep our upper bound small enough, as events of the form $\{X_u(s) \geq v s\}$ have small probability but $Y_t(A,R,K)$ is large on that event. Using the Girsanov transform, \eqref{eqn:momentII} yields
\begin{align*}
  &\E(Y_t(A,R,K))\\
  \leq &\alpha \int_R^t \E\left( e^{-\theta \sigma B_s}F(t-s,\sigma B_s+v s) \ind{\sigma B_s \leq K} \right) \dd s \\
  \leq &C\alpha e^{\sqrt{2} A} \int_R^t e^{-\sqrt{2} (\sqrt{2} - v) s} \left( \frac{t+1}{t-s+1} \right)^{\frac{3}{2}} \E\left( e^{(\sqrt{2} -\theta)\sigma B_s} \left( 1 + (v + \sqrt{2})s + (-B_s)_+\right) \ind{B_s \leq K} \right)\dd s,
\end{align*}
using \eqref{eqn:boundPhiII}. As $(\beta,\sigma^2) \in \mathcal{C}_{II}$ and $\sigma^2 > 1$, we have $\beta \sigma^2 < 1$. This yields in particular $\beta < \sigma^2$ hence $\sqrt{2}-\theta> 0$. Integrating with respect to the Brownian density, there exists $C>0$ such that
\[
  \E\left( e^{(\sqrt{2} -\theta)\sigma B_s} \left( 1 + \sqrt{2} (\sqrt{\beta \sigma^2} +1)s + (-B_s)_+\right) \ind{B_s \leq K} \right) \leq C (1 + s)^{\frac{1}{2}} e^{(\sqrt{2} -\theta)\sigma K},
\]
yielding
\[
  \E(Y_t(A,R,K)) \leq C \alpha e^{\sqrt{2} A + (\sqrt{2} -\theta)\sigma K} \int_R^t e^{-2(1 - \sqrt{\beta \sigma^2}) s} \left( \frac{t+1}{t-s+1} \right)^{\frac{3}{2}}(1 + s)^{\frac{1}{2}} \dd s.
\]
Then by dominated convergence, as $1 - \sqrt{\beta \sigma^2}>0$, we deduce that
\[
  \limsup_{t \to \infty} \E(Y_t(A,R,K)) \leq C \alpha e^{\sqrt{2} A + (\sqrt{2} -\theta)\sigma K} \int_R^\infty e^{-2(1 - \sqrt{\beta \sigma^2}) s} (1 + s)^{\frac{1}{2}} \dd s,
\]
which decreases to $0$ as $R \to \infty$, completing the proof.
\end{proof}

We now use the known asymptotic behaviour of the extremal process of the branching Brownian motion, recalled in Section \ref{sec:facts}, to compute the asymptotic behaviour of the extremal process of particles satisfying $T(u) \leq R$, defined as
\[
  \hat{\mathcal{E}}^R_t := \sum_{u \in \mathcal{N}^2_t} \ind{T(u) \leq R} \delta_{X_u(t) - m^{(II)}_t}.
\] 

For any $u \in \mathcal{B}$, and $t \geq 0$, we set
\[
  Z^{(u)}_t := \sum_{\substack{u' \in \mathcal{N}^2_t\\u' \succcurlyeq u}} (\sqrt{2} t - X_{u'}(t)) e^{\sqrt{2} (X_{u'}(t) - \sqrt{2} t)},
\]
where we recall that $u' \succcurlyeq u$ denotes that $u'$ is a descendant of $u$. Note that by \eqref{eqn:cvDerivative} and the branching property, $Z^{(u)}_t$ converges a.s. to the variable
$\displaystyle
  Z^{(u)}_\infty := \liminf_{t \to \infty} Z^{(u)}_t.
$
Moreover, $e^{-\sqrt{2}(X_u(T(u)) - \sqrt{2} T(u))} Z^{(u)}_\infty \egaldistr Z_\infty$, where $Z_\infty$ is the limit of the derivative martingale of a standard branching Brownian motion.

\begin{lemma}
\label{lem:extremalII}
For all $\phi \in \mathcal{T}$, we have $\displaystyle\lim_{t \to \infty} \crochet{\hat{\mathcal{E}}^R_t,\phi} = \crochet{\hat{\mathcal{E}}^R_\infty,\phi}$ in law,
where $\hat{\mathcal{E}}^R_\infty$ is a decorated Poisson point process with intensity $c_\star \bar{Z}_R \sqrt{2}e^{-\sqrt{2}x} \dd x$ and decoration law $\mathfrak{D}$, with $\displaystyle \bar{Z}_R := \sum_{u \in \mathcal{B}} \ind{T(u) \leq R} Z^{(u)}_\infty$.
\end{lemma}

\begin{proof}
Let $\phi \in \mathcal{T}$ be a test function. Observe that using the branching property of the branching Brownian motion, we have
\[
  \E\left( \exp\left( -\crochet{\hat{\mathcal{E}}^R_t, \phi} \right) \right) = \E\left( \prod_{u \in \mathcal{B} : T(u) \leq R} F_t(T(u),X_u(T(u))) \right),
\]
where $F_t(s,x) = \E^{(2)}\left( \exp\left( -\sum_{u \in \mathcal{N}_t} \phi\left(x+X_u(t-s)-m_t^{(II)}\right) \right) \right)$ for $0 \leq s \leq t$ and $x \in \R$. Using again that $m_t^{(II)} = m_{t-s}^{(II)} + \sqrt{2 } s + o(1)$ as $t \to \infty$ and applying Lemma \ref{lem:abbs}, we have for all $s \geq 0$
\[
  \lim_{t \to \infty}  F_t(s,x) = \E^{(2)} \left( \exp\left(- c_\star Z_\infty e^{\sqrt{2} x - 2 s} \int (1 - e^{-\Psi[\phi] (z)})\sqrt{2}e^{- \sqrt{2 } z} \dd z  \right)\right),
\]
where $Z_\infty$ is the limit of the derivative martingale in a standard branching Brownian motion. Therefore, by dominated convergence theorem,
\[
  \lim_{t \to \infty}  \E\left( \exp\left( -\crochet{\hat{\mathcal{E}}^R_t, \phi} \right) \right)
  = \E\left( \exp\left( -c_\star \bar{Z}_R \int (1 - e^{-\Psi[\phi](z)})\sqrt{2 }e^{- \sqrt{2} z} \dd z \right) \right),
\]
with $\bar{Z}_R = \sum_{u \in \mathcal{B}} \ind{T(u) \leq R} Z_\infty^{(u)}$, completing the proof.
\end{proof}

We then observe that $\hat{\mathcal{E}}^R_\infty$ converges in law as $R \to \infty$ to $\mathcal{E}^{(II)}_\infty$ the point measure defined in Theorem~\ref{thm:mainII}.
\begin{lemma}
\label{lem:limitII}
For all $\phi \in \mathcal{T}$, we have $\lim_{R \to \infty} \crochet{\hat{\mathcal{E}}^R_\infty,\phi} = \crochet{\mathcal{E}^{(II)}_\infty, \phi}$ in law, where $\bar{Z}_\infty := \sum_{u \in \mathcal{B}} Z^{(u)}_\infty$.
\end{lemma}

\begin{proof}
Recall that $Z_\infty \geq 0$ a.s. therefore $(\bar{Z}_R, R \geq 0)$ is increasing and $\bar{Z}_\infty = \lim_{R \to \infty} \bar{Z}_R$ exists a.s. Given that for all function $\phi \in \mathcal{T}$,
\[
  \E\left( e^{-\crochet{\hat{\mathcal{E}}^R_\infty,\phi}} \right) =\E\left( \exp\left( -c_\star \bar{Z}_R \int (1 - e^{-\Psi[\phi](z)})\sqrt{2}e^{- \sqrt{2} z} \dd z \right) \right),
\]
to prove that $\mathcal{E}^R_\infty$ converges in law, it is enough to show that $\bar{Z}_\infty < \infty$ a.s.

We recall that $v = \sqrt{2\beta \sigma^2}$ the speed of the branching Brownian motion of particles of type $1$. As
\[ \lim_{t \to \infty} \max_{u \in \mathcal{N}^1_t} X_u(t) - v t = -\infty \quad \text{a.s.}\]
(which is a consequence of the fact that the additive martingale at the critical parameter converges to 0 a.s.), there is almost surely finitely many $u \in \mathcal{B}$ with $X_u(T(u)) \geq v T(u)$. To prove the finiteness of $\bar{Z}_\infty$, we then use the following variation on Kolmogorov's three series theorem: If we have
\begin{equation}
  \sum_{u \in \mathcal{B}} \ind{X_u(T(u)) \leq v T(u)} \E\left( Z_\infty^{(u)} \wedge 1\middle|\mathcal{F}^1 \vee \sigma(\mathcal{B}) \right) < \infty \quad \text{a.s.}, \label{eqn:aimLimitMean}
\end{equation}
then $\tilde{Z}_\infty := \sum_{u \in \mathcal{B}} \ind{X_u(T(u)) \leq v T(u)}  Z_\infty^{(u)}  < \infty$ a.s, where we recall that $\mathcal{F}^1 = \sigma(X_u(t), u \in \mathcal{N}_t, t \geq 0)$. Using that $\bar{Z}_\infty$ is obtained by adding a finite number of finite random variables to $\tilde{Z}_\infty$, it implies that $\bar{Z}_\infty < \infty$ a.s

Indeed, if we assume \eqref{eqn:aimLimitMean}, using the Markov inequality and the Borel-Catelli lemma, we deduce that almost surely there are finitely many $u \in \mathcal{B}$ whose contribution to $\tilde{Z}_\infty$ is larger than $1$. Additionally, \eqref{eqn:aimLimitMean} also implies that the sum of all the other contributions to $\tilde{Z}_\infty$ has finite mean. Hence, we have $\tilde{Z}_\infty < \infty$ a.s.

We now prove \eqref{eqn:aimLimitMean}, using that by \eqref{eqn:tailDerivativeMartingale} and \eqref{eqn:asymptoticsDerivativeMartingale} for all $x \in \R$, we have $\E((Z_\infty e^x) \wedge 1) \leq C (1 + (-x)_+) e^x$. Hence, using that $Z^{(u)}_\infty\egaldistr e^{\sqrt{2}(X_u(T(u)) - \sqrt{2} T(u))} Z_\infty$  it is enough to show that
\begin{equation}
  \sum_{u \in \mathcal{B}} \ind{X_u(T(u)) \leq v T(u)}\left( 1 + \left(2  T(u) - \sqrt{2} X_u(T(u))\right)_+\right)e^{\sqrt{2} X_u(T(u)) - 2 {T(u)}}< \infty \quad \text{a.s.}
  \label{eqn:simplifyAim}
\end{equation}
This quantity being a series of positive random variables, we prove that this series has finite mean to conclude.
By Corollary \ref{cor:poissonSum}, we have
\begin{multline*}
  \E\left(   \sum_{u \in \mathcal{B}} \ind{X_u(T(u)) \leq v {T(u)}}\left( 1 + \left(2  {T(u)} - \sqrt{2 } X_u(T(u))\right)_+\right)e^{\sqrt{2} X_u(T(u)) - 2 {T(u)}} \right)\\
  = \alpha \int_0^\infty e^{\beta s}\E\left( \ind{\sigma B_s \leq v s} \left(1 + \left( 2 s - \sqrt{2} \sigma B_s \right)_+  \right) e^{\sqrt{2} \sigma B_s - 2 s} \right) \dd s.
\end{multline*}
Similarly to the proof of Lemma \ref{lem:originII}, we bound the above quantity in two different ways depending on whether $\sigma^2 > 1$ or $\sigma^2 \leq 1$.

If $\sigma^2 \leq 1$, we have
\begin{equation*}
   e^{\beta s}\E\left( \left(1 + \left( \sqrt{2 } \sigma B_s  - 2 s\right)_+  \right) e^{\sqrt{2} \sigma B_s - 2 s} \right)
  \leq  C (1 + s) \exp\left( s \left(\sigma^2 + \beta - 2 \right) \right),
\end{equation*}
which decays exponentially fast as $(\beta,\sigma^2) \in \mathcal{C}_{II}$ and $\sigma^2 \leq 1$. Therefore, we have
\[
  \E\left(   \sum_{u \in \mathcal{B}} \ind{X_u(T(u)) \leq v {T(u)}}\left( 1 + \left(2 T(u) - \sqrt{2} X_u(T(u))\right)_+\right)e^{\sqrt{2 } X_u(T(u)) - 2  {T(u)}} \right) < \infty,
\]
proving \eqref{eqn:simplifyAim}, hence \eqref{eqn:aimLimitMean}, therefore that $\bar{Z}_\infty < \infty$ a.s. in that case.

If $\sigma^2 > 1$, we have
\begin{multline*}
   e^{\beta s}\E\left( \ind{\sigma B_s \leq v s}\left(1 + \left( 2s - \sqrt{2} \sigma B_s \right)_+  \right) e^{\sqrt{2} \sigma B_s - 2 s} \right)\\
  = \E\left( \ind{B_s \leq 0} \left(1 + \left( 2(1 - \sqrt{\sigma^2\beta}) s - \sqrt{2} \sigma B_s \right)_+  \right) e^{\sqrt{2} \sigma B_s} \right) e^{2 \left(\sqrt{\sigma^2 \beta} - 1 \right)s}\\
  \leq  C (1 + s) \exp\left( s \left( \sqrt{\sigma^2 \beta } -1 \right) \right).
\end{multline*}
As $(\beta,\sigma^2) \in \mathcal{C}_{II}$ and $\sigma^2 < 1$ we have once again
\[
  \E\left(   \sum_{u \in \mathcal{B}} \ind{X_u(T(u)) \leq v {T(u)}}\left( 1 + \left( 2  {T(u)} - \sqrt{2} X_u(T(u))\right)_+ \right)e^{\sqrt{2} X_u(T(u)) - 2   {T(u)}} \right) < \infty,
\]
which proves that $\bar{Z}_\infty < \infty$ a.s. in that case as well.
\end{proof}

Using the above results, we finally obtain the asymptotic behaviour of the extremal process in case $\mathcal{C}_{II}$. 
\begin{proof}[Proof of Theorem \ref{thm:mainII}]
Recall that $\hat{\mathcal{E}}_t = \sum_{u \in \mathcal{N}^2_t} \delta_{X_u(t) - m_t}$. Using Proposition \ref{prop:convergencePointProcess}, we only need to prove that for all $\phi \in \mathcal{T}$, we have
\[
  \lim_{t \to \infty} \E\left( e^{-\crochet{\hat{\mathcal{E}}_t,\phi}} \right) = \E\left( e^{-\crochet{\mathcal{E}^{(II)}_\infty,\phi}} \right).
\]
Let $\phi \in \mathcal{T}$, and set $A \in \R$ such that $\phi(z) = 0$ for all $z \leq A$. By Lemma \ref{lem:originII}, for all $\epsilon>0$, there exists $R \geq 0$ such that
$
  \P\left( \hat{\mathcal{E}}^R_t(\phi) \neq \hat{\mathcal{E}}_t(\phi) \right) \leq \epsilon.
$
Then, using that $\phi$ is non-negative, so that $\crochet{\hat{\mathcal{E}}^R_t,\phi} \leq \crochet{\hat{\mathcal{E}}_t,\phi}$ and $e^{-\crochet{ \hat{\mathcal{E}}_t,\phi}}$ is bounded by $1$, we have $\displaystyle \E\left( e^{-\crochet{\hat{\mathcal{E}}^R_t,\phi}} \right) \leq \E\left( e^{-\crochet{\hat{\mathcal{E}}_t,\phi}} \right) \leq \E\left( e^{-\crochet{\hat{\mathcal{E}}^R_t,\phi}} \right) + \epsilon$. Applying Lemma~\ref{lem:extremalII} and Lemma \ref{lem:limitII} to let $t$, then $R$, grow to $\infty$, we obtain
\[
  \E\left( e^{-\crochet{\mathcal{E}^{(II)}_\infty,\phi}} \right) \leq \liminf_{t \to \infty}\E \left(e^{-\crochet{\hat{\mathcal{E}}_t,\phi}} \right) \leq \limsup_{t \to \infty} \E \left(e^{-\crochet{\hat{\mathcal{E}}_t,\phi}} \right) \leq \E\left( e^{-\crochet{\mathcal{E}^{(II)}_\infty,\phi}} \right) + \epsilon.
\]
Letting $\epsilon \to 0$ we obtain that $\lim_{t \to \infty}\E \left(e^{-\crochet{\hat{\mathcal{E}}_t,\phi}} \right) = \E\left( e^{-\crochet{\mathcal{E}^{(II)}_\infty,\phi}} \right)$ for all $\phi \in \mathcal{T}$, which completes the proof of Theorem \ref{thm:mainII} by Remark \ref{rem:15}.
\end{proof}

We end this section by conjecturing a possible direct formula for the computation of $\bar{Z}_\infty$ as the limit of a sub-martingale of the multitype BBM.
\begin{conjecture}
We have $\lim_{t \to \infty} \sum_{u \in \mathcal{N}_t^2} (\sqrt{2} t - X_u(t)) e^{\sqrt{2} X_u(t) - 2t} = \bar{Z}_\infty$ a.s.
\end{conjecture}

\section{Proof of Theorem \ref{thm:mainI}}
\label{sec:proofI}

In this section, we assume that $(\beta,\sigma^2) \in \mathcal{C}_I$, that is either $\sigma^2 \leq 1$ and $\sigma^2 > \frac{1}{\beta}$, or $\sigma^2 > 1$ and $\sigma^2 >  \frac{\beta}{2\beta-1}$. In that situation, we show that the extremal process of particles of type $2$ is mainly driven by the asymptotic of particles of type $1$, and that any particle of type $2$ significantly contributing to the extremal process at time $t$ satisfies $t - T(u) = O(1)$, meaning that they have a close ancestor of type $1$.

For the rest of the section, we denote by $v = \sqrt{2\beta \sigma^2}$ and $\theta = \sqrt{2\beta/\sigma^2}$ the speed and critical parameter of the BBM of particles of type $1$. Recall that $m^{(I)}_t = v t - \frac{3}{2\theta} \log t$, and we set
\[
  \hat{\mathcal{E}}_t := \sum_{u \in \mathcal{N}^2_t} \delta_{X_u(t) - m^{(I)}_t},
\]
the extremal process of particles of type $2$, centred around $m^{(I)}_t$.

To prove Theorem~\ref{thm:mainI}, we first show that for all $\phi \in \mathcal{T}$, $\crochet{\hat{\mathcal{E}}_t,\phi}$ converges, as $t \to \infty$ to a proper random variable. By \cite[Lemma 5.1]{Kal}, this is enough to conclude that $\hat{\mathcal{E}}_t$ converges vaguely in law to a limiting point measure $\bar{\mathcal{E}}$. We then use that with high probability, no particle of type $2$ born before time $R$ contributes, to the extremal process of the multitype BBM. Then, by the branching property, it shows that $\bar{\mathcal{E}}$ satisfies a stability under superposition probability which, by \cite[Corollary~3.2]{Mai13}, can be identified as a decorated Poisson point process with intensity proportional to $Z^{(I)}_\infty e^{-\theta x} \dd x$.

To prove the results of this section, we make use of the following extension of \eqref{eqn:triangleShape}. For all $t \geq 0$, we write $a_t = \frac{3}{2\theta} \log(t+1)$. There exists $C>0$ such that for all $t \geq 0$ and $K>0$, we have
\begin{equation}
  \label{eqn:triangleImproved}
  \P\left(\exists s \leq t, u \in \mathcal{N}^1_s : X_u(s) \geq v s - a_t + a_{t-s} + K\right)\leq C (K+1)e^{-\theta K}.
\end{equation}
This result was proved in \cite{Mal15a} in the context of branching random walks, and has been adapted to continuous-time settings in \cite[Lemma 3.1]{Mal15c}.

We first show the tightness of the law of the number of particles of type $2$ born to the right of $m^{(I)}_t-A$.
\begin{lemma}
\label{lem:computation}
We assume that $(\beta,\sigma^2) \in \mathcal{C}_I$. For all $A,K>0$, there exists $C_{A,K} > 0$ and $\delta > 0$ such that for all $R \geq 0$, we have
\[
  \limsup_{t \to \infty} \E\left( \sum_{u \in \mathcal{N}^2_t} \ind{X_u(t) \geq m^{(I)}_t -A} \ind{T(u) \leq t-R} \ind{X_u(s) \leq vs - a_t + a_{t-s} + K, s \leq T(u)}  \right) \leq C_{A,K} e^{-\delta R}.
\]
\end{lemma}

\begin{proof}
Let $A,K,R  >0$, we set
\[
  Y_t(A,K,R) = \sum_{u \in \mathcal{N}^2_t} \ind{X_u(t) \geq m^{(I)}_t -A} \ind{T(u) \leq t-R} \ind{X_u(s) \leq vs - a_t + a_{t-s} + K, s \leq T(u)}.
\]
We use Proposition \ref{prop:manytoone2type} to compute the mean of $Y_t(A,K,R)$ as
\begin{align*}
  \E\left( Y_t(A,K,R) \right)
  \leq & \int_0^{t-R} e^{\beta s + t-s} \P\left( \sigma B_s + B_t-B_s \geq m^{(I)}_t - A, \sigma B_r \leq vr - a_t + a_{t-r} + K, r \leq s \right)  \dd s\\
  \leq & \int_0^{t-R} \E\left( e^{-\theta \sigma B_s} F(t-s,\sigma B_s - a_t) \ind{\sigma B_r \leq a_{t-r} - a_t + K, r \leq s} \right) \dd s,
\end{align*}
using the Markov property at time $s$ and the Girsanov transform, where $F(r,x) = e^r\P(B_r \geq v r - x )$. By the exponential Markov inequality, for all $\lambda > 0$, we have $F(r,x) \leq e^{\lambda x} e^{r \left( 1 - \lambda v + \tfrac{\lambda^2}{2} \right)}$. This implies
\begin{equation}
  \label{eqn:intermediate}
  \E(Y_t(A,K,R)) \leq \int_0^{t-R} e^{(t-s)(1 - \lambda v + \tfrac{\lambda^2}{2})} (t+1)^{\tfrac{3\lambda}{2\theta}} \E\left( e^{\sigma (\lambda - \theta) B_s} \ind{\sigma B_r \leq a_{t-r} - a_t + K, r \leq s}  \right) \dd s.
\end{equation}
We now bound this quantity in two different ways depending on the sign of $\sigma^2 - 1$.

First, if $\sigma^2 > 1$, then $v > \theta$, in which case using \eqref{eqn:intermediate} with $\lambda = v$, we obtain
\[
  \E(Y_t(A,K,R)) \leq \int_0^{t-R} e^{(t-s)(1 - v^2/2)} (t+1)^{\tfrac{3v}{2\theta}} \E\left( e^{\sigma (v - \theta) B_s} \ind{\sigma B_r \leq a_{t-r} - a_t + K, r \leq s}  \right) \dd s.
\]
We now use that for all $\lambda > 0$, there exists $C > 0$ such that for all $0 \leq s \leq t$, we have
\begin{equation}
  \label{eqn:intermediateStep}
  \E\left( e^{\lambda \sigma B_{s}} \ind{\sigma B_r \leq a_{t-r} - a_{t} + K, r \leq s}  \right) \leq C e^{\lambda K} \left( \tfrac{t-s+1}{t+1} \right)^{\frac{3\lambda }{2\theta}} (s+1)^{-\tfrac{3}{2}}.
\end{equation}
This bound can be obtained by classical Gaussian estimates, rewriting
\begin{multline*}
  \E\left( e^{\lambda \sigma B_{s}} \ind{\sigma B_r \leq a_{t-r} - a_{t} + K, r \leq s}  \right)\\
  \leq Ce^{\lambda K}\left( \tfrac{t-s+1}{t+1} \right)^{\frac{3\lambda}{2\theta}} \sum_{k \geq 0} e^{-\lambda k} \P(\sigma B_{s} - a_{s} + a_{t} + K \in [-k-1,-k], \sigma B_r \leq a_{t-r} - a_{t} + K, r \leq s),
\end{multline*}
and showing that the associated probability can be bounded uniformly in $k, t$ and $s \leq t$ by $C k (s+1)^{-\tfrac{3}{2}}$, with computations similar to the ones used in \cite[Lemma~3.8]{Mal15a} for random walks. Therefore, \eqref{eqn:intermediateStep} implies that
\[
  \E(Y_t(A,K,R)) \leq C_{A,K} \int_0^{t-R} e^{(t-s)(1 - v^2/2)} \frac{(t+1)^{\tfrac{3}{2}} (t-s+1)^{\tfrac{3v}{2\theta}}}{(s+1)^{\tfrac{3}{2}}}  \dd s.
\]
As $(\beta,\sigma^2) \in \mathcal{C}_I$, we have $v = \sqrt{2\beta\sigma^2} > \sqrt{2}$, so $1 - \tfrac{v^2}{2} < 0$. As a result
\[
  \int_0^{\tfrac{t}{2}} e^{(t-s)(1 - v^2/2)} \frac{(t+1)^{\tfrac{3}{2}} (t-s+1)^{\tfrac{3v}{2\theta}}}{(s+1)^{\tfrac{3}{2}}}  \dd s \leq C e^{(1 - v^2/2)t} (t+1)^{\tfrac{5 \theta + 3v}{2\theta}},
\]
which converges to $0$ as $t \to \infty$, and
\[
  \int_{\tfrac{t}{2}}^{t-R} e^{(t-s)(1 - v^2/2)} \frac{(t+1)^{\tfrac{3}{2}} (t-s+1)^{\tfrac{3v}{2\theta}}}{(s+1)^{\tfrac{3}{2}}}  \dd s \leq C \int_R^\infty e^{(1-v^2/2)s} (s+1)^{\tfrac{3v}{2\theta}} \dd s \leq C e^{(1-v^2/2)R/2}.
\]
This completes the proof of the lemma in the case $\sigma^2 > 1$.

We now assume that $\sigma^2<1$. We have that
\[
  1 - \theta v + \frac{\theta^2}{2} = 1 - 2 \beta + \frac{\beta}{\sigma^2} =  \beta \left( \tfrac{1}{\sigma^2} - 2\right) + 1.
\]
Therefore, as long as $\sigma^2 > \frac{\beta}{2 \beta-1}$, which is the case as $\sigma^2 < 1$ and $(\beta,\sigma^2) \in \mathcal{C}_I$, we have $1 - \theta v + \tfrac{\theta^2}{2} < 0$. Therefore, for all $\delta > 0$ small enough such that
\[
  1 - (\theta+\delta) v + \tfrac{(\theta+\delta)^2}{2} < 0,
\]
using \eqref{eqn:intermediate} with $\lambda = \theta + \delta$, we have
\[
  \E(Y_t(A,K,R)) \leq \int_0^{t-R} e^{(t-s)(1 - (\theta+\delta) v + (\theta+\delta)^2/2)} (t+1)^{\tfrac{3(\theta+\delta)}{2\theta}} \E\left( e^{\sigma \delta B_s} \ind{\sigma B_r \leq a_{t-r} - a_t + K, r \leq s}  \right) \dd s.
\]
So with the same computations as above, we obtain once again that
\[
  \limsup_{t \to \infty} \E(Y_t(A,K,R)) \leq C_{A,K} e^{-\delta R},
\]
which completes the proof.
\end{proof}

Using the above computation, we deduce that with high probability, only particles of type $2$ having an ancestor of type $1$ at time $t-O(1)$ contribute substantially to the extremal process at time $t$.
\begin{lemma}
\label{lem:originI}
Assuming that $(\beta,\sigma^2) \in \mathcal{C}_{I}$, for all $A > 0$, we have
\[
  \lim_{R \to \infty} \limsup_{t \to \infty} \P(\exists u \in \mathcal{N}^2_t : T(u)\leq  t-R, X_u(t)\geq m_t^{(I)}-A) = 0.
\]
\end{lemma}

\begin{proof}
We observe that for all $K>0$, we have
\begin{multline*}
   \P(\exists u \in \mathcal{N}^2_t : T(u)\leq  t-R, X_u(t)\geq m_t^{(I)}-A)
   \leq \P(\exists s \leq t, u \in \mathcal{N}^1_s : X_u(s) \geq vs + a_t - a_{t-s})\\
   +\P(\exists u \in \mathcal{N}^2_t : T(u)\leq  t-R, X_u(t)\geq m_t^{(I)}-A, X_u(s) \leq v s - a_t + a_{t-s}, s \leq T(u)).
\end{multline*}
Then, using \eqref{eqn:triangleImproved}, the Markov inequality and Lemma \ref{lem:computation}, we obtain
\[
  \limsup_{t \to \infty} \P(\exists u \in \mathcal{N}^2_t : T(u)\leq  t-R, X_u(t)\geq m_t^{(I)}-A) \leq C (K+1)e^{-\theta K} + C_{A,K} e^{-R}.
\]
Letting $R \to \infty$ then $K \to \infty$, the proof is now complete.
\end{proof}

For all $R > 0$, we set
\[
  \hat{\mathcal{E}}^R_t := \sum_{u \in \mathcal{N}^2_t} \ind{T(u) \geq t-R} \delta_{X_u(t) - m^{(I)}_t}.
\]
We now show that $\hat{\mathcal{E}}^R_t$ converges in law as $t \to \infty$.
\begin{lemma}
\label{lem:convI}
Assume that $(\beta,\sigma^2) \in \mathcal{C}_I$, there exists $c_R > 0$ and a point measure distribution $\mathcal{D}^R$ such that for all $\phi \in \mathcal{T}$, we have
\[
  \lim_{t \to \infty} \crochet{\hat{\mathcal{E}}^R_t,\phi} = \crochet{\hat{\mathcal{E}}^R_\infty,\phi} \quad \text{ in law,}
\]
where $\hat{\mathcal{E}}^R_\infty$ is a DPPP($c_R Z_\infty e^{-\theta x} \dd x$,$\mathfrak{D}^R$).
\end{lemma}

\begin{proof}
We can rewrite
\begin{equation*}
  \hat{\mathcal{E}}^R_t = \sum_{u \in \mathcal{N}^1_{t-R}} \sum_{\substack{u' \in \mathcal{N}^2_t\\u' \succcurlyeq u}} \delta_{X_{u'}(t) - X_{u}(t-R) + X_u(t-R) - m^{(I)}_t}  = \sum_{u \in \mathcal{N}^1_{t-R}} \tau_{X_u(t-R) - m^{(I)}_t} \hat{\mathcal{E}}^{(u)}_R,
\end{equation*}
where $\tau_z$ is the operator of translation by $z$ of point measures, and $\hat{\mathcal{E}}^{(u)}_R$ is the point process of descendants of type $2$ of individual $u \in \mathcal{N}_{t-R}$ at time $t$, centred around the position of $u$ at time $t-R$. Note that conditionally on $\mathcal{F}^1_{t-R}$, $(\hat{\mathcal{E}}^{(u)}_R, u \in \mathcal{N}_{t-R})$ are i.i.d. point measures with same law as
$
  \bar{\mathcal{E}}_R := \sum_{u \in \mathcal{N}^2_R} \delta_{X_u(R)}.
$

Let $\phi \in \mathcal{T}$, we set $L \in \R$ such that $\phi(x) = 0$ for all $x \leq L$. By the branching property, we have
\begin{align*}
  \E\left( e^{-\crochet{\hat{\mathcal{E}}^R_t,\phi}} \right) = \E\left( \prod_{u \in \mathcal{N}^1_{t-R}} F_R(X_u(t-R)-m^{(I)}_t) \right)
  = \E\left( e^{-\sum_{u \in \mathcal{N}^1_{t-R}} -\log F_R(X_u(t-R)-m^{(I)}_t)} \right),
\end{align*}
where $F_R(x) = \E\left( \exp\left(- \sum_{u \in \mathcal{N}^2_R} \phi(x + X_u(R))\right) \right)$. Observe that by Jensen transform, we have
\begin{align*}
  - \log F_R(x) &\leq \E\left( \sum_{u \in \mathcal{N}^2_R} \phi(x + X_u(R)) \right) \leq  \int_0^R e^{\beta s + (R-s)} \E(\phi(x + \sigma B_s + B_t - B_s)) \dd s\\
  &\leq ||\phi||_\infty R e^{(\beta +1)R} \P\left(B_1 \geq \tfrac{-x}{\sqrt{R(\sigma^2 + 1)}}\right).
\end{align*}
Therefore, by \eqref{eqn:gaussianEstimate}, we have $-\log F_R(x) \leq C_R e^{(\theta + \delta)x} \wedge 1$ for all $x \in \R$.

By Lemma \ref{lem:abbs}, recall that $\sum_{u \in \mathcal{N}^1_{t-R}} \delta_{X_u(t-R) - m^{(I)}_t}$ converges vaguely in law to a DPPP $\mathcal{E}^1$ with intensity $c_\star \theta Z_\infty e^{-\theta (z+vR)} \dd z$ and decoration law $\mathfrak{D}_{\beta,\sigma^2}$ the law of the decoration point measure of the BBM with branching rate $\beta$ and variance $\sigma^2$. Additionally, it was proved by Madaule \cite{Mad17} in the context of branching random walks, and extended in \cite{CHL19} to BBM settings, that $\crochet{\mathcal{E}^1,e_{\theta + \delta}} < \infty$ a.s. for all $\delta > 0$, where $e_{\theta + \delta}(x) = e^{(\theta + \delta)x}$. As a result, using the monotone convergence theorem, we obtain
\[
  \lim_{t \to \infty} \E\left( e^{-\crochet{\hat{\mathcal{E}}^R_t,\phi}} \right) = \E\left( e^{-\crochet{\mathcal{E}^1,-\log F_R}} \right).
\]
This proves that $\hat{\mathcal{E}}^R_t$ converges in law, as $t \to \infty$, to a point process that can be obtained from $\mathcal{E}^1$ by replacing each atom of $\mathcal{E}^1$ by an independent copy of the point measure $\bar{\mathcal{E}}_R$.
\end{proof}

We now complete the proof of Theorem~\ref{thm:mainI}.
\begin{proof}[Proof of Theorem~\ref{thm:mainI}]
Let $\phi \in \mathcal{T}$. We fix $A > 0$ such that $\phi(x) = 0$ for all $x \leq -A$. We observe that
\[
   0 \leq \E\left( e^{- \crochet{\hat{\mathcal{E}}^R_t,\phi}} \right) - \E\left( e^{-\crochet{\hat{\mathcal{E}}_t,\phi}} \right) \leq \P(\exists u \in \mathcal{N}^2_t : T(u)\leq  t-R, X_u(t)\geq m_t^{(I)}-A),
\]
which goes to $0$ as $t$ then $R \to \infty$, by Lemma \ref{lem:originI}. Additionally, by Lemma \ref{lem:convI}, we have
\[
  \lim_{t \to \infty} \E\left( e^{-\crochet{\hat{\mathcal{E}}^R_t,\phi}} \right) = \E\left( e^{-\crochet{\hat{\mathcal{E}}^R_\infty,\phi}} \right).
\]
Moreover, using that $R \mapsto \E\left( e^{-\crochet{\hat{\mathcal{E}}^R_t,\phi}} \right)$ is decreasing, we deduce that
\begin{equation}
  \label{eqn:approximation}
  \lim_{t \to \infty}  \E\left( e^{- \crochet{\hat{\mathcal{E}}_t,\phi}} \right) = \lim_{R\to \infty} \E\left( e^{-\crochet{\hat{\mathcal{E}}^R_\infty,\phi}} \right).
\end{equation}
Additionally, as $R \mapsto \hat{\mathcal{E}}^R_t$ is increasing in the space of point measures, we observe that we can construct the family of point measures $(\hat{\mathcal{E}}^R_\infty, R \geq 0)$ on the same probability space in such a way that almost surely, $\crochet{\hat{\mathcal{E}}^R_\infty, \phi}$ is increasing for all $\phi$. We denote by $\mu(\phi)$ its limit.

By \cite[Lemma 5.1]{Kal}, to prove that $\hat{\mathcal{E}}_t$ admits a limit in distribution for the topology of vague convergence, it is enough to show that for all non-negative continuous functions with compact support, $\crochet{\hat{\mathcal{E}}_t,\phi}$ admits a limit in law which is a proper random variable. By \eqref{eqn:approximation}, using the monotonicity of $\hat{\mathcal{E}}^R_\infty$, we immediately obtain that $\lim_{t \to \infty} \crochet{\hat{\mathcal{E}}_t,\phi} = \mu(\phi)$ in law. Therefore, to prove that $\hat{\mathcal{E}}_t$ converges vaguely in distribution, it is enough to show that for all $\phi \in \mathcal{T}$, $\mu(\phi) < \infty$ a.s. which is a consequence of the tightness of $\crochet{\hat{\mathcal{E}}_t,\phi}$.

Let $\phi \in \mathcal{T}$, we write $L \in \R$ such that $\phi(x) =0$ for all $x < L$. For all $A > 0$ and $K>0$, we have
\begin{multline*}
  \P\left(\crochet{\hat{\mathcal{E}}_t,\phi} \geq A\right) \leq \P(\exists s \leq t, \exists u \in \mathcal{N}_s : X_u(s) \geq vs - a_t + a_{t-s}+K)\\
  + \frac{1}{A}\E\left( \crochet{\hat{\mathcal{E}}_t,\phi} \ind{\max_{u \in \mathcal{N}^1_s }X_u(s) \leq vs -a_t+a_{t-s}+K, s \leq t} \right).
\end{multline*}
The first quantity goes to $0$ as $K \to \infty$ by \eqref{eqn:triangleImproved}. Therefore, for all $\epsilon> 0$, we can fix $K$ large enough so that it remains smaller than $\epsilon/2$. Then, using Lemma \ref{lem:computation} with $R = 0$, we have
\[
  \frac{1}{A}\E\left( \crochet{\hat{\mathcal{E}}_t,\phi} \ind{\max_{u \in \mathcal{N}^1_s }X_u(s) \leq vs -a_t+a_{t-s}+K, s \leq t} \right) \leq \frac{C_{L,K}}{A}.
\]
Therefore, we can choose $A$ large enough such that for all $t \geq 0$, $\P(\crochet{\hat{\mathcal{E}}_t,\phi} \geq A) \leq \epsilon$, which completes the proof of the tightness of $\crochet{\hat{\mathcal{E}}_t,\phi}$.

We then conclude that $\hat{\mathcal{E}}_t$ converges vaguely in law as $t \to \infty$ to a limiting point measure that we write $\bar{\mathcal{E}}$. This point measure also is the limit as $R \to \infty$ of $\hat{\mathcal{E}}^R_\infty$, by \eqref{eqn:approximation}. This allows us to show that $\crochet{\hat{\mathcal{E}}_t,\phi} \to \crochet{\bar{\mathcal{E}},\phi}$ in law for all $\phi \in \mathcal{T}$, so we conclude by Proposition~\ref{prop:convergencePointProcess} that the position of the rightmost atom in $\hat{\mathcal{E}}_t$ also converges to the position of the rightmost atom in $\bar{\mathcal{E}}$.

To complete the proof of Theorem~\ref{thm:mainI}, we have to describe the law of $\bar{\mathcal{E}}$. For all $s \geq 0$, using the branching property, we have
\[
  \hat{\mathcal{E}}_t = \sum_{u \in \mathcal{N}_s} \tau_{X_u(s) - vs + a_{t-s} - a_t} \hat{\mathcal{E}}^{(u)}_{t-s}
\]
where conditionally on $\mathcal{F}_s$, $(\hat{\mathcal{E}}^{(u)}_{t-s}, u \in \mathcal{N}_s)$ is a family of independent point measures with same law as $\hat{\mathcal{E}}_{t-s}$, under law $\P^{(1)}$ or $\P^{(2)}$ depending on the type of $u$. As no particle of type $2$ born at an early time will have a descendant contributing in the extremal process by Lemma~\ref{lem:originI}, we obtain that, letting $t \to \infty$,
\begin{equation}
  \label{eqn:branchingProp}
  \bar{\mathcal{E}} \egaldistr \sum_{u \in \mathcal{N}^1_s} \tau_{X_u(s) - vs} \bar{\mathcal{E}}^{(u)},
\end{equation}
where $\bar{\mathcal{E}}^{(u)}$ are i.i.d. copies of $\bar{\mathcal{E}}$, that are further independent of $\mathcal{F}_s$. This superposition property characterizes the law of $\bar{\mathcal{E}}$ as a decorated Poisson point process with intensity proportional to $e^{-\theta x}\dd x$, shifted by the logarithm of the derivative martingale of the branching Brownian motion by \cite[Corollary 3.2]{Mai13}, with similar computations as in \cite[Section 2.2]{Mad17}. A general study of such point measures satisfying the branching property \eqref{eqn:branchingProp} is carried out in \cite{MM}.
\end{proof}

\section{Asymptotic behaviour in the anomalous spreading case}
\label{sec:anomalous}

We assume in this section that $(\sigma^2,\beta) \in \mathcal{C}_{III}$, i.e. that $\beta + \sigma^2 > 2$ and $\sigma^2 < \frac{\beta}{2 \beta -1}$. In particular, it implies that $\beta > 1$ and $\sigma^2 < 1$. Under these conditions, we set
\[
  \theta := \sqrt{2 \frac{\beta - 1 }{1-\sigma^2}}, \quad a := \sigma^2 \theta, \quad b := \theta \quad \text{and} \quad p := \frac{\sigma^2 + \beta - 2}{2(\beta-1)(1-\sigma^2)},
\]
which are the values of $a$, $b$ and $p$ solutions of \eqref{eqn:optimizationProblem}, described in terms of the parameter $\theta$ which plays the role of a Lagrange multiplier in the optimization problem. Note that $a < \sqrt{2 \beta \sigma^2}$, $b > \sqrt{2}$ and $p \in (0,1)$. Recall that in this situation, the maximal displacement is expected to satisfy
\[
  m^{(III)}_t = v t, \quad \text{where }\  v = a p + b (1-p) = \frac{\beta - \sigma^2}{\sqrt{2 (\beta - 1)(1 - \sigma^2)}}.
\]
As in the previous sections, we set
\[
  \hat{\mathcal{E}}_t = \sum_{u \in \mathcal{N}_t^{2}} \delta_{X_u(t) - m^{(III)}_t}
\]
the appropriately centred extremal process of particles of type $2$.

As mentioned in Section~\ref{subsec:stateoftheart}, under the above assumption, we are in the anomalous behaviour regime. In this regime, we have $v >\max(\sqrt{2},\sqrt{2\beta\sigma^2})$, in other words, this furthest particle travelled at a larger speed than the ones observed in the BBM of particles of type $1$, or in a BBM of particles of type $2$. Moreover, given the heuristic explanation for \eqref{eqn:optimizationProblem}, we expect the furthest particle $u$ of type $2$ at time $t$ to satisfy $T(u) \approx pt$ and $X_u(T(u)) \approx a p t$.

The idea of the proof of Theorem~\ref{thm:mainIII} is to show that this heuristic holds, and that all particles participating to the extremal process of the multitype BBM are of type~$2$, and satisfy $T(u) \approx p t$ and $X_{u}(T(u)) \approx a pt$. We then use the asymptotic behaviour of the growth rate of the number of particles of type $1$ growing at speed $a$ to complete the proof.
We begin by proving that with high probability, there is no particle of type $2$ far above level $m^{(III)}_t$ at time~$t$. 
\begin{lemma}
\label{lem:firstStepIII}
Assuming that $(\sigma^2,\beta) \in \mathcal{C}_{III}$, we have
$\displaystyle 
  \lim_{A \to \infty} \limsup_{t \to \infty} \P\left(\exists u \in \mathcal{N}^2_t : X_u(t)\geq m_t^{(III)}+A\right) = 0.
$
\end{lemma}

\begin{proof}
The proof of this result is based on a first moment method. For $A> 0$, we compute, using the many-to-one lemma, the mean of
$
  X_t(A) = \sum_{u \in \mathcal{N}_t^2} \ind{X_u(t) \geq m_t^{(III)}+A}.
$
Using \eqref{eqn:gaussianEstimate}, there exists $C>0$ such that for all $t \geq 1$, we have
\begin{align*}
  \E(X_t(A))
  &= \int_0^t e^{\beta s + t-s} \P\left(\sigma B_s + (B_t- B_s) \geq m_t^{(III)}+A\right) \dd s\\
  &\leq \int_0^t e^{\beta s + (t-s)} \frac{C\sqrt{\sigma^2 s + t-s}}{(vt + A)} e^{- \frac{(v t + A )^2}{2(\sigma^2 s + t-s)}}  \dd s \leq C t^{-1/2} \int_0^t e^{\beta s +(t-s)} e^{- \frac{(v t + A )^2}{2(\sigma^2 s + t-s)}} \dd s.
\end{align*}
Therefore, setting $\phi  : u \mapsto \beta u + 1-u - \frac{v^2}{2(\sigma^2 u + 1-u)}$, by change of variable we have, for all $t$ large enough
\begin{equation*}
  \E(X_t(A)) \leq C t^{1/2} \int_0^1 \exp\left( t \phi(u) \right) e^{- A \frac{v}{(\sigma^2 u + 1-u)}} \dd u.
\end{equation*}

We observe that
\begin{equation*}
  \phi'(u) = \beta - 1 - (1 - \sigma^2) \frac{v^2}{2 (\sigma^2 u + 1-u)^2}\quad \text{ and }\quad
  \phi''(u) = -2 (1 - \sigma^2)^2 \frac{v^2}{2(\sigma^2 u + 1-u)^3},
\end{equation*}
hence $\phi$ is concave, and maximal at point $u =p$, with a maximum equal to $0$. By Taylor expansion, there exists $\delta > 0$ such that $\phi(u) \leq -\delta (u-p)^2$ for all $u \in [0,1]$. Therefore, we have
\[
  \E(X_t(A)) \leq C e^{-A v}  t^{1/2} \int_0^1 e^{-\delta (u-p)^2 t} \dd u \leq C e^{-A v} \sqrt{\pi/\delta}.
\]

As a result, applying the Markov inequality, we have
\[
  \P\left(\exists u \in \mathcal{N}^2_t : X_u(t)\geq m_t^{(III)}+A\right) = \P(X_t(A) \geq 1) \leq \E(X_t(A)),
\]
thus there exists $C > 0$ such that
\[
  \limsup_{t \to \infty}  \P\left(\exists u \in \mathcal{N}^2_t : X_u(t)\geq m_t^{(III)}+A\right) \leq C e^{-A v},
\]
which converges to $0$ as $A \to \infty$.
\end{proof}

Next, we show that every particle of type $2$ that contributes to the extremal process of the BBM branched from a particle of type $1$ at a time and position close to $(pt,apt)$.
\begin{lemma}
\label{lem:originIII}
Assuming that $(\sigma^2,\beta) \in \mathcal{C}_{III}$, for all $A > 0$, we have
\begin{align}
  \label{eqn:equation1}
  &\lim_{R \to \infty} \limsup_{t \to \infty} \P(\exists u \in \mathcal{N}^2_t : X_u(t)\geq m_t^{(III)}-A, |T(u) - pt| \geq R t^{1/2}) = 0,\\
  \label{eqn:equation2}
  \text{and} \quad &\lim_{R \to \infty} \limsup_{t \to \infty} \P(\exists u \in \mathcal{N}^2_t : X_u(t)\geq m_t^{(III)}-A, |X_u(T(u)) - apt| \geq R t^{1/2}) = 0.
\end{align}
\end{lemma}

\begin{proof}
Let $A > 0$ and $\epsilon > 0$. By Lemma \ref{lem:firstStepIII}, there exists $K>0$ such that with probability $(1-\epsilon)$ no particle of type $2$ is above level $m^{(III)}_t + K$ at time $t$ for all $t$ large enough. For $R > 0$, we now compute the mean of
\[
  Y^{(1)}_t(A,K,R) = \sum_{u \in \mathcal{N}_t^2} \ind{X_u(t) - m_t^{(III)} \in [-A,K]} \ind{|T(u) - pt| \geq R t^{1/2}}.
\]
By Proposition \ref{prop:manytoone2type}, setting $I_t(R) = [0,t]\backslash [pt-R t^{1/2},pt+Rt^{1/2}]$ we have
\begin{align*}
  \E\left(Y^{(1)}_t(A,K,R)\right) &= \int_{I_t(R)} e^{\beta s + (t-s)} \P\left( \sigma B_s + (B_t-B_s) - m_t^{(III)}  \in [- A,K]\right) \dd s\\
  &\leq C e^{Av} t^{1/2} \left( \int_0^{p - Rt^{-1/2}} e^{t \phi(u)} \dd u + \int_{p + Rt^{-1/2}}^1 e^{t \phi(u)} \dd u \right),
\end{align*}
using the same notation and computation techniques as in the proof of Lemma \ref{lem:firstStepIII}. Thus, using again that there exists $\delta > 0$ such that $\phi(u) \leq - \delta(u-p)^2$ for some $\delta > 0$, by change of variable $z = t^{1/2}(u-p)$ we obtain that
\[
  \E\left(Y^{(1)}_t(A,K,R)\right) \leq C e^{Av} \int_{\R\setminus [-R,R]} e^{- \delta z^2} \dd z.
\]
Therefore, by Markov inequality, we obtain that
\begin{multline*}
  \limsup_{t \to \infty}  \P(\exists u \in \mathcal{N}^2_t : X_u(t)\geq m_t^{(III)}-A, |T(u) - pt| \geq R t^{1/2}) \\
  \leq \limsup_{t \to \infty} \P(\exists u \in \mathcal{N}^2_t : X_u(t)\geq m_t^{(III)}+K) + C e^{Av} \int_{\R\setminus [-R,R]} e^{- \delta z^2} \dd z.
\end{multline*}
As a result, with the choice previously made for the constant $K$, we obtain that
\[
  \limsup_{R\to\infty} \limsup_{t \to \infty} \P(\exists u \in \mathcal{N}^2_t : X_u(t)\geq m_t^{(III)}-A, |T(u) - pt| \geq R t^{1/2}) \leq \epsilon.
\]
By letting $\epsilon \to 0$, we complete the proof of \eqref{eqn:equation1}.

We now turn to the proof of \eqref{eqn:equation2}. By \eqref{eqn:equation1}, we can assume, up to enlarging the value of $K$ that
\[
  \limsup_{t \to \infty} \P(\exists u \in \mathcal{N}^2_t : X_u(t)\geq m_t^{(III)}-A, |T(u) - pt| \geq K t^{1/2}) \leq \epsilon. 
\]
We now compute the mean of
\[
  Y^{(2)}_t(A,K,R) = \sum_{u \in \mathcal{N}_t^2} \ind{X_u(t) - m_t^{(III)} \in [-A,K]} \ind{|T(u) - pt| \leq K t^{1/2}} \ind{ |X_u(T(u)) - a pt| \geq R t^{1/2}}.
\]
Using again Proposition \ref{prop:manytoone2type}, we have
\begin{align*}
  &\E\left( Y^{(2)}_t(A,K,R) \right)\\
  = &\int_{pt - K t^{1/2}}^{pt + Kt^{1/2}} e^{\beta s + t-s} \P\left( \sigma B_s + (B_t-B_s) - m_t^{(III)} \in [-A,K], |\sigma B_s - a pt| \geq R t^{1/2} \right)\dd s\\
  = &\int_{pt - K t^{1/2}}^{pt + Kt^{1/2}} e^{2 (\beta - 1)(s- pt)} \E\left( e^{\theta (\sigma B_s + B_t-B_s)} \ind{ \begin{array}{l}\scriptstyle \sigma B_s + (B_t-B_s)  + (b-a)(pt-s) \in [-A,K]\\ \scriptstyle |\sigma B_s - a(pt-s)| \geq R t^{1/2}\end{array}}  \right) \dd s,
\end{align*}
by Girsanov transform, using that $\beta s + t-s - \frac{\theta^2}{2}(\sigma^2 s + t-s) = 2 (\beta - 1)(pt-s)$ and straightforward computations. Next, using that $\theta (b-a)(pt-s) = - 2 (\beta - 1)(s-pt)$, for $R$ large enough, we obtain
\begin{multline*}
 \E\left( Y^{(2)}_t(A,K,R) \right)\\
  \leq 2 K t^{1/2}e^{\theta K} \sup_{|r| \leq K t^{1/2}} \P\left(  \sigma B_{pt + r} + B_t-B_{pt+r}  + (b-a)r \in [-A,K], |\sigma B_{pt+r} - a r| \geq R t^{1/2}  \right).
\end{multline*}

Then, by classical Gaussian computations, $\sigma B_{pt+r} - \left(\sigma B_{pt + r} + B_t-B_{pt+r}\right)\frac{\sigma^2 pt + r}{\sigma^2(pt + r) + t(1 - p) - r} $ is independent of $\sigma B_{pt + r} + B_t-B_{pt+r}$. We deduce that for $R$ large enough, we have for all $t$ large enough
\begin{multline*}
  \sup_{|r| \leq K t^{1/2}} \P\left(  \sigma B_{pt + r} + B_t-B_{pt+r}  + (b-a)r \in [-A,K], |\sigma B_{pt+r} - a r| \geq R t^{1/2}  \right)\\
  \leq C(A+K) t^{-1/2} \P\left(|B_1| \geq \frac{R}{2} \sqrt{\tfrac{ (\sigma^2 p + 1-p)}{4 \sigma^2 p(1-p)}}\right).
\end{multline*}
As a result, using again the Markov inequality, we have
\begin{multline*}
  \limsup_{t \to \infty}  \P(\exists u \in \mathcal{N}^2_t : X_u(t)\geq m_t^{(III)}-A, |X_u(T(u)) - apt| \geq R t^{1/2}) \\
  \leq \limsup_{t \to \infty} \P(\exists u \in \mathcal{N}^2_t : X_u(t)\geq m_t^{(III)}-A, |T(u) - pt| \geq K t^{1/2}) \qquad \qquad \qquad \qquad \\
  + \limsup_{t \to \infty} \P(\exists u \in \mathcal{N}^2_t : X_u(t)\geq m_t^{(III)}+K)
  + C(A+K) K e^{\theta K} \P\left(|B_1| \geq R \sqrt{\tfrac{ (\sigma^2 p + 1-p)}{4 \sigma^2 p(1-p)}}\right).
\end{multline*}
Hence, letting $R \to \infty$, with the choice made for the constant $K$, we obtain
\[
  \limsup_{R\to\infty} \limsup_{t \to \infty}  \P(\exists u \in \mathcal{N}^2_t : X_u(t)\geq m_t^{(III)}-A, |X_u(T(u)) - apt| \geq R t^{1/2}) \leq 2 \epsilon,
\]
and letting $\epsilon \to 0$ completes the proof of \eqref{eqn:equation2}.
\end{proof}

The above lemma shows that typical particles of type $2$ that contribute to the extremal process of the multitype BBM have their last ancestor of type $1$ around time $pt$ and position $pat$. We now prove Theorem~\ref{thm:mainIII}, using this localization of birth times and positions of particles in $\mathcal{B}$ that have a descendant contribution to the extremal process at time $t$, with high probability. Then, using Lemmas \ref{lem:cltExpanded} and \ref{lem:poissonProc} we compute the quantity of contributing particles and with Lemma \ref{lem:jointLargeDev} to obtain the value associated to each contribution.

\begin{proof}[Proof of Theorem \ref{thm:mainIII}]
Let $R > 0$, we set
\[
  \hat{\mathcal{E}}^R_t := \sum_{u \in \mathcal{N}^2_t} \ind{|T(u) - pt| \leq Rt^{1/2}, |X_u(T(u)) - apt| \leq R t^{1/2}} \delta_{X_u(t) - m^{(III)}_t}.
\]
Lemma \ref{lem:originIII} states that the extremal process $\hat{\mathcal{E}}^R_t$ is close to the extremal process of the BBM. Precisely, for all $\phi \in \mathcal{T}$ we have
\begin{multline*}
   \left| \E\left(e^{-\crochet{\hat{\mathcal{E}}^R_t,\phi}}\right) - \E\left( e^{-\crochet{\hat{\mathcal{E}}_t,\phi}} \right) \right|\\
    \leq \P\left(\exists u \in \mathcal{N}^2_t : X_u(t)\geq m_t^{(III)}-A , \ (T(u)-pt,X_u(T(u))-apt) \not \in [- Rt^{1/2},Rt^{1/2}]^2\right)
\end{multline*}
where $A$ is such that the support of $\phi$ is contained in $[-A,\infty)$. As a result, by Lemma \ref{lem:originIII} we have
\begin{equation}
  \label{eqn:observation}
  \lim_{R \to \infty} \limsup_{t \to \infty} \left| \E\left(e^{-\crochet{\hat{\mathcal{E}}^R_t,\phi}}\right) - \E\left( e^{-\crochet{\hat{\mathcal{E}}_t,\phi}} \right) \right| = 0,
\end{equation}
so to compute the asymptotic behaviour of $ \E\left( e^{-\crochet{\hat{\mathcal{E}}_t,\phi}} \right)$, it is enough to study the convergence of $\hat{\mathcal{E}}^R_t$ as $t$ then $R$ grow to $\infty$.

Let $R > 0$ and $\phi \in \mathcal{T}$. Using the branching property and Corollary \ref{cor:poissonSum}, we have
\[
  \E\left( e^{ -\crochet{\hat{\mathcal{E}}^R_t,\phi}} \right) = \E\left( \exp\left( - \alpha \int_{pt - Rt^{1/2}}^{pt+Rt^{1/2}} \sum_{u \in \mathcal{N}^1_s} \ind{|X_u(s)-apt| \leq R t^{1/2}}F(t-s,X_u(s)-apt) \dd s \right) \right),
\]
with $F(r,x) = 1 - \E^{(2)}\left( e^{- \sum_{u \in \mathcal{N}^2_r} \phi(X_u(r)- b r +x- b((1-p)t-r))} \right)$. Additionally, by Lemma \ref{lem:jointLargeDev}, we have
\begin{align*}
  F((1 - p)t - r, x) &= C(b)\frac{e^{(1 - \frac{b^2}{2})((1 - p)t-r)}}{\sqrt{2\pi t(1-p)}}  e^{b (x-br) - \tfrac{(x-br)^2}{2((1-p)t-r)}} \int e^{-\theta z} \left(1 - e^{-\Psi^b[\phi](z)}\right) \dd z (1 + o(1))\\
  &= C(b)\frac{e^{(1 - \frac{b^2}{2})(1 - p)t}}{\sqrt{2\pi t(1-p)}}  e^{\theta x -(1 + \frac{\theta^2}{2})r - \tfrac{(x-br)^2}{2(1-p)t}} \int e^{-\theta z} \left(1 - e^{-\Psi^b[\phi](z)}\right) \dd z (1 + o(1)),
\end{align*}
as $t \to \infty$, uniformly in $|r|\leq R t^{1/2}$ and $|x| \leq Rt^{1/2}$, where we used that $\theta = b$. Thus, setting
\[
  \Theta(\phi) := \alpha C(b) \int e^{-\theta z} \left(1 - e^{-\Psi^{b}[\phi](z)}\right) \dd z
\]
and $G_R(r,x) = \ind{|x+ar|\leq R} e^{-\frac{(x + (a-b)r)^2}{2(1-p)}}$ we can now rewrite the Laplace transform of $\hat{\mathcal{E}}^R_t$ as
\begin{align*}
  &\E\left( e^{ -\crochet{\hat{\mathcal{E}}^R_t,\phi}} \right)\\ = &\E\left( \exp\left( - \Theta(\phi)(1+o(1))\tfrac{e^{(1 - \frac{b^2}{2})(1 - p)t}}{\sqrt{2\pi t(1-p)}}  \!\!\int_{pt- Rt^{1/2}}^{pt+Rt^{1/2}}\!\!\!\! \sum_{u \in \mathcal{N}^1_{s}} e^{\theta (X_u(s) - apt) -(1 + \frac{\theta^2}{2})(s-pt)} G_R\left(\tfrac{s-pt}{t^{1/2}},\tfrac{X_u(s)- as}{t^{1/2}}\right) \dd s \right) \right).
\end{align*}
By definition of the parameters, we have $\beta + \frac{\theta^2\sigma^2}{2} = 1 + \frac{\theta^2}{2} = \frac{\beta- \sigma^2}{1-\sigma^2}$ and $\theta a = \beta + \frac{\theta^2 \sigma^2}{2} - \left( \beta - \frac{a^2}{2\sigma^2} \right)$, therefore we can rewrite
\begin{multline}
  \label{eqn:steppp}
  \E\left( e^{ -\crochet{\hat{\mathcal{E}}^R_t,\phi}} \right)\\ = \E\left( \exp\left( - \Theta(\phi)\frac{1+o(1)}{\sqrt{2\pi t(1-p)}}  \int_{pt- Rt^{1/2}}^{pt+Rt^{1/2}} \sum_{u \in \mathcal{N}^1_{s}} e^{\theta X_u(s) - (\beta + \frac{\theta^2\sigma^2}{2})s} G_R\left(\tfrac{s-pt}{t^{1/2}},\tfrac{X_u(s)- as}{t^{1/2}}\right) \dd s \right) \right),
\end{multline}
where we used that $(1 -p) (1 - \frac{b^2}{2}) +p(1 - \frac{a^2}{2\sigma^2}) = 0$.

We now observe that by Lemma \ref{lem:cltExpanded}, using \eqref{eqn:scaling}, we have
\begin{multline*}
  \lim_{t \to \infty} \frac{1}{t^{1/2}}  \int_{pt- Rt^{1/2}}^{pt+Rt^{1/2}} \sum_{u \in \mathcal{N}^1_{s}} e^{\theta X_u(s) - (\beta + \frac{\theta^2\sigma^2}{2})s} G_R\left(\tfrac{s-pt}{t^{1/2}},\tfrac{X_u(s)- apt}{t^{1/2}}\right) \dd s\\
   = \frac{W_\infty(\theta)}{\sqrt{2\pi p \sigma^2}} \int_{[-R,R] \times \R} e^{-\frac{z^2}{2\sigma^2 p}} e^{- \frac{(z+(a-b)r)^2}{2(1 -p)}} \ind{|z + ar| \leq R} \dd r \dd z
\end{multline*}
where $W_\infty(\theta) = \lim_{t \to \infty} \sum_{u \in \mathcal{N}^1_t} e^{\theta X_u(t) - t(\beta + \frac{\sigma^2 \theta^2}{2})}$ is the limit of the additive martingale with parameter $\theta$ for the branching Brownian motion of type $1$. As a result, writing
\[
  c_R = \frac{1}{2\pi\sqrt{p(1-p) \sigma^2}}\int_{[-R,R] \times \R} e^{-\frac{z^2}{2\sigma^2 p}} e^{- \frac{(z+(a-b)r)^2}{2(1 -p)}} \ind{|z + ar| \leq R} \dd r \dd z \in (0,\infty),
\]
by dominated convergence theorem, \eqref{eqn:steppp} yields
\[
  \lim_{t \to \infty}\E\left( e^{ -\crochet{\hat{\mathcal{E}}^R_t,\phi}} \right) = \E\left( \exp\left( - c_R W_\infty(\theta) \Theta(\phi) \right) \right).
\]
This convergence holds for all $\phi \in \mathcal{T}$. Then by \cite[Lemma 4.4]{BBCM19}, the process $\hat{\mathcal{E}}_t^R$ converges vaguely in distribution as $t \to \infty$ to a DPPP($\theta {c}_R W_\infty(\theta) e^{-\theta z}\dd z,\mathfrak{D}^b$), as $t \to \infty$, where $\mathfrak{D}^b$ is the law of $\mathcal{D}^b$, the point measure defined in \eqref{eqn:defineSupercriticalDecoration}.

To complete the proof, we now observe that by monotone convergence theorem, we have
\[
  \lim_{R \to \infty} c_R = \frac{1}{2\pi\sqrt{p(1-p) \sigma^2}}\int_{\R \times \R} e^{-\frac{z^2}{2\sigma^2 p}} e^{- \frac{(z+(a-b)r)^2}{2(1 -p)}} \dd r = \frac{1}{b-a} = \frac{1}{\theta (1 - \sigma^2)}.
\]
Therefore, letting $t \to \infty$ then $R \to \infty$, \eqref{eqn:observation} yields
\[
  \lim_{t \to \infty} \E\left(e^{-\crochet{\hat{\mathcal{E}}_t,\phi}}\right) =  \E\left( \exp\left( - \frac{\alpha C(b) W_\infty(\theta)}{2 (\beta - 1)}\int \theta e^{-\theta z} \left(1 - e^{-\Psi^{b}[\phi](z)}\right) \right) \right).
\]
As a result, using \cite[Lemma 4.4]{BBCM19}, the proof of Theorem \ref{thm:mainIII} is now complete, with $c^{(III)} = \frac{\alpha C(b)}{2(\beta-1)}$ and $\mathfrak{D}^{(III)}$ the law of $\mathcal{D}^b$ defined in \eqref{eqn:defineSupercriticalDecoration}.
\end{proof}

\paragraph{Acknowledgements.}
The authors are partially funded by ANR-16-CE93-0003 (ANR MALIN). Additionally, M.A.B. is partially supported by Cofund MathInParis project from FSMP. 

\noindent\begin{minipage}{0.1\textwidth}
\includegraphics[width=\textwidth]{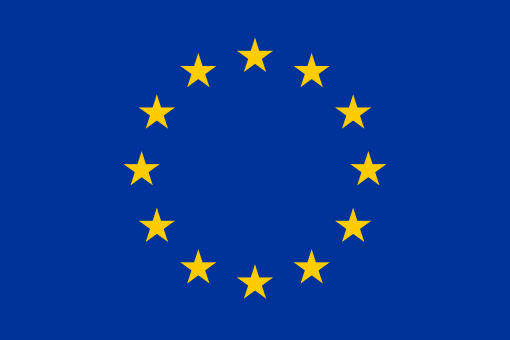}	
\end{minipage}\quad
\begin{minipage}{0.85\textwidth}
This program has received funding from the European Union's Horizon 2020 research and innovation programme under the Marie Skłodowska-Curie grant agreement No 754362.
\end{minipage}


\end{document}